\documentclass[12pt]{amsart}

\usepackage[margin=3 cm]{geometry}
\usepackage[stable]{footmisc}
\usepackage{graphicx}
\usepackage[shortlabels]{enumitem}
\usepackage{cite}
\usepackage{amssymb,amsmath,mathrsfs}
\usepackage{amsthm}
\usepackage{multirow,multicol}
\usepackage{hyperref}
\usepackage{booktabs}
\usepackage{subcaption}
\usepackage{cleveref}
\usepackage{xspace}

\newcommand{\NN}{{\mathbb N}}
\newcommand{\RR}{{\mathbb R}}

\renewcommand{\SS}{{\mathbb S}}
\newcommand{\bSS}{\mathcal{BS}_m^N}
\renewcommand{\vec}{\mathrm{vec}}
\newcommand{\lamI}{{\boldsymbol\lambda^{\rm I}}}
\newcommand{\lamII}{{\boldsymbol\lambda^{\rm II}}}
\newcommand{\id}{I}
\newcommand{\JI}{{{J}_{\rm I}}}
\newcommand{\JII}{{{J}_{\rm II}}}
\newcommand{\Mdef}{{\mathcal{M}_{\rm def}}}
\newcommand{\cat}{\mathbin{\Vert}}

\newcommand\refzero{{\hyperref[eq:nearestpoint]{\upshape($\mathrm{P}_{\bartheta}$)}}}

\DeclareMathOperator{\Sym}{bSym} 
\DeclareMathOperator{\symm}{sym} 
\DeclareMathOperator{\rank}{rank}

\DeclareMathOperator{\trace}{tr}
\DeclareMathOperator{\image}{Im}
\DeclareMathOperator{\spann}{span}
\DeclareMathOperator{\val}{val}

\newcommand{\slra}{\texttt{slra}\xspace}
\newcommand{\mosek}{\texttt{Mosek}\xspace}
\newcommand{\cvx}{\texttt{CVX}\xspace}

\newcommand{\ba}{{\boldsymbol a}}
\newcommand{\bb}{{\boldsymbol b}}
\newcommand{\bc}{{\boldsymbol c}}
\newcommand{\bd}{{\boldsymbol d}}
\newcommand{\be}{{\boldsymbol e}}
\newcommand{\bbf}{{\boldsymbol f}\mkern-1mu}
\newcommand{\bg}{{\boldsymbol g}}
\newcommand{\bl}{{\boldsymbol l}}
\newcommand{\bp}{{\boldsymbol p}}
\newcommand{\bs}{{\boldsymbol s}}
\newcommand{\bu}{{\boldsymbol u}}
\newcommand{\bv}{{\boldsymbol v}}
\newcommand{\bw}{{\boldsymbol w}}
\newcommand{\bx}{{\boldsymbol x}}
\newcommand{\by}{{\boldsymbol y}}
\newcommand{\bz}{{\boldsymbol z}}
\newcommand{\blambda}{{\boldsymbol \lambda}}
\newcommand{\btheta}{{\boldsymbol \theta}}
\newcommand{\bzeta}{{\boldsymbol \zeta}}
\newcommand{\barv}{{\boldsymbol{\bar v}}}
\newcommand{\barw}{{\boldsymbol{\bar w}}}
\newcommand{\barx}{{\boldsymbol{\bar x}}}
\newcommand{\bary}{{\boldsymbol{\bar y}}}
\newcommand{\barz}{{\boldsymbol{\bar z}}}
\newcommand{\bartheta}{\boldsymbol{\bar\theta}}
\newcommand{\vone}{\boldsymbol{\tilde{v}}}

\newcommand*\pct{\scalebox{.9}{\%}}

\renewcommand{\subset}{\subseteq}
\renewcommand{\supset}{\supseteq}

\newtheorem{theorem}{Theorem}
\numberwithin{theorem}{section}

\newtheorem{lemma}[theorem]{Lemma}
\newtheorem{assumption}{Assumption}

\theoremstyle{definition}

\newtheorem{example}[theorem]{Example}
\theoremstyle{remark}
\newtheorem{remark}[theorem]{Remark}
\let\oldtabular\tabular
\renewcommand{\tabular}{\footnotesize\oldtabular}

\Crefname{assumption}{Assumption}{Assumptions}

\graphicspath{ {../images/} }

\title[A convex relaxation for the nearest structured rank deficient matrix]{A convex relaxation to compute the \\nearest structured rank deficient matrix}
\date{\today}
\author{Diego Cifuentes}
\address{
Department of Mathematics,
Massachusetts Institute of Technology, Cambridge MA 02139, USA}
\email{diegcif@mit.edu}


\begin{document}
\maketitle

\begin{abstract}
Given an affine space of matrices $\mathcal{L}$ and a matrix $\Theta\in \mathcal{L}$, consider the problem of computing the closest rank deficient matrix to $\Theta$ on $\mathcal{L}$ with respect to the Frobenius norm.
This is a nonconvex problem with several applications in 
control theory, computer algebra, and computer vision.
We introduce a novel semidefinite programming (SDP) relaxation,
and prove that it always gives the global minimizer of the nonconvex problem in the low noise regime,
i.e., when $\Theta$ is close to be rank deficient.
Our SDP is the first convex relaxation for this problem with provable guarantees.
We evaluate the performance of our SDP relaxation in examples from
system identification, approximate GCD, triangulation, and camera resectioning.
Our relaxation reliably obtains the global minimizer under non-adversarial noise,
and its noise tolerance is significantly better than state of the art methods.

\end{abstract}

\section{Introduction}\label{s:fractional}

Let $k,m,n$ be positive integers with $m\!\leq\! n$.
Given
an affine map $\mathscr{S}: \RR^k\to \RR^{m\times n}$
and a vector $\btheta\in\RR^k$,
we consider the nearest point problem
\begin{equation}\label{eq:stls}
\begin{aligned}
  \min_{\bu\in \mathcal{U}}\quad & \|\bu-\btheta\|^2,
  \quad \text{ with }\quad
  \mathcal{U} := \{\bu\in \RR^k : \mathscr{S}(\bu) \text{ is rank deficient}\},
\end{aligned}
\end{equation}
where we use the $\ell_2$-norm in~$\RR^k$.
More generally, we may consider the weighted $\ell_2$-(semi)norm
$\|\bv\|_W = \sqrt{\bv^T W \bv}$,
where~$W$ is positive (semi)definite.
Denoting the affine space $\mathcal{L} := \image \mathscr{S}$
and the matrix $\Theta := \mathscr{S}(\btheta)$,
we may equivalently write~\eqref{eq:stls} as
\begin{equation}\label{eq:stls1}\tag{1'}
\begin{aligned}
  \min_{{U}\in \RR^{m\times n}}\quad
  & \|{U}-\Theta\|^2,
  \quad \text{ such that }\quad
  {U} \in \mathcal{L}
  \quad\text{ and }\quad
  {U} \text{ is rank deficient},
\end{aligned}
\end{equation}
where we use the (possibly weighted) Frobenius norm in $\RR^{m\times n}$.

We may further study a variant of~\eqref{eq:stls} with complex variables,
but this can be reduced to the real case.
Indeed, a complex matrix $U$ is rank deficient if and only if
the real matrix
$
  \left(\begin{smallmatrix}
    \Re(U) & - \Im(U) \\ \Im(U) & \Re(U)
  \end{smallmatrix}\right)
$
is rank deficient.
So we restrict our attention to the real case.

Problem~\eqref{eq:stls} is sometimes known as \emph{structured total least squares} (STLS)
\cite{Markovsky2007,DeMoor1993,Rosen1996,Lemmerling1999,Lemmerling2001}.
The name comes from the total least squares method
for the linear regression problem $A \bx \approx \bb$,
given by minimizing the Frobenius norm $\|(\Delta A \,|\, \Delta \bb)\|$
subject to $(A{+}\Delta A \,|\, \bb{+}\Delta \bb)$ being rank deficient.
By further imposing an affine constraint on the matrix, we arrive to~\eqref{eq:stls}.

There are several applications of~\eqref{eq:stls} in control, systems theory, and statistics
\cite{Markovsky2007,Markovsky2008,Markovsky2011,Chu2003,DeMoor1993,Kaltofen2006}.
Lemmerling's thesis \cite{Lemmerling1999} gives a long list of concrete applications,
including noisy deconvolution, linear prediction, direction of arrival estimation, information retrieval, approximate realization, and $H_2$~model reduction.
The approximate GCD problem~\cite{Kaltofen2006} from computer algebra is another instance of~\eqref{eq:stls}.
There are also many applications in computer vision.
Specifically, in the area of multi-view geometry
one often has to minimize the following \emph{fractional programming} problem:
\begin{equation}\label{eq:fractional}
\begin{aligned}
  \min_{\bz\in\RR^m,\bu\in\RR^k}\quad & \|\bu - \btheta\|^2,
  \quad \text{ such that }\quad
  u_i = \frac{\ba_i^T \bz}{\bb_i^T \bz} &\text{ for }i=1,\dots,k,
\end{aligned}
\end{equation}
for some given vectors $\ba_i,\bb_i \!\in\! \RR^m$.
After clearing denominators,
this takes the form of~\eqref{eq:stls} with $n\!=\!k$.
In particular, the triangulation, camera resectioning, and homography estimation problems can be phrased as above~\cite{Kahl2008}.

Problem~\eqref{eq:stls} is nonconvex and computationally hard.
It is usually solved in practice with local optimization methods.
Although these methods are typically fast,
there are no guarantees of convergence to the global optimum.
An alternative approach to solve nonconvex problems is to construct a convex relaxation.
This idea has been applied in several other areas,
e.g., \cite{Yang2020,Candes2009,Rosen2019,Iguchi2017,Cosse2017},
often leading to global optimality results.
We follow this approach in this paper and propose a novel convex relaxation for~\eqref{eq:stls}.

An important advantage of convex relaxations over local optimization methods is \emph{certifiability}.
Even if a local optimization method successfully finds the global minimum,
there is no way for the user to verify this.
On the other hand, convex relaxations may leverage convex duality to provide certificates of global optimality.
Guaranteed performance is a critical feature in certain applications;
its role in robotics is discussed in \cite[Appx.A]{Yang2020}.
Convex relaxations might also be combined with local optimization methods,
serving for a posteriori verification of candidate solutions.

Our convex relaxation, presented in~\eqref{eq:sdp},
relies on \emph{semidefinite programming} (SDP),
and can be solved in polynomial time with interior points methods~\cite[\S6]{BenTal2001}.
The optimal value of our relaxation is always a lower bound on the global minimum of~\eqref{eq:stls}.
We say that the relaxation is \emph{exact} (or tight) if the optimal solution of the SDP is a rank-one matrix.
In such a case the optimal value of the SDP agrees with the optimal value of~\eqref{eq:stls}
and we can recover the global minimizer of~\eqref{eq:stls} from the SDP solution.

We evaluated the performance of our SDP relaxation in the following applications:
approximate realization (control),
approximate GCD (computer algebra),
triangulation and camera resectioning (computer vision).
Our experiments reveal that our SDP is often exact (and certifiably optimal) under non-adversarial noise.
Furthermore, the SDP is significantly more resilient to noise than state of the art methods.
Preliminary experimental results suggest that the SDP is also robust against missing data (accommodated by using a degenerate weight matrix~$W$).

\subsection*{Exactness under low noise}

Our main technical contribution is to show that our SDP relaxation of~\eqref{eq:stls} is always {exact} when $\btheta$ is \emph{sufficiently close} to the set~$\mathcal{U}$.
In statistical estimation problems the point $\btheta$ often represents a noisy sample from~$\mathcal{U}$,
so the distance from $\btheta$ to $\mathcal{U}$ corresponds to the magnitude of the noise.
We hence refer to this setting as the \emph{low noise regime}.
This gives a theoretical explanation for the practical success of the relaxation in our experiments.

More precisely, we show that if ${\bartheta}\in \mathcal{U}$ satisfies a mild geometric condition,
then the SDP relaxation is exact as $\btheta\to {\bartheta}$.
The geometric condition we use is \emph{non-tangentiality},
as introduced in~\cite{Andersson2013}.
Two manifolds $\mathcal{M}_1, \mathcal{M}_2$ meet {non-tangentially} at a point~$\bp$ if
\begin{align}\label{eq:nontangential}
  \mathcal{M}_1 \!\cap\! \mathcal{M}_2 \text{ is a manifold nearby } \bp
  \quad\text{ and }\quad
  T_{\bp}(\mathcal{M}_1 \!\cap\! \mathcal{M}_2) =
  T_{\bp}(\mathcal{M}_1) \cap T_{\bp}(\mathcal{M}_2),
\end{align}
where $T_{\bp}$ denotes the tangent space at~$\bp$.
Informally, this condition means that the two manifolds have a positive angle in directions perpendicular to $\mathcal{M}_1 \cap \mathcal{M}_2$.
As shown in~\cite{Andersson2013}, non-tangentiality is a weaker condition than the more common notion of transversality
(which means that
$T_{\bp}(\mathcal{M}_1) \!+\! T_{\bp}(\mathcal{M}_2)$ is the whole space).

\begin{assumption} \label{ass:transversality}
  Let
  $\Mdef := \{A : \rank A \!=\! m{-}1 \} \subset \RR^{m\times n}$
  be the manifold of matrices with rank deficiency~one.
  Assume that $\Mdef$ meets the affine space $\mathcal{L} \!=\! \image \mathscr{S}$
  non-tangentially at the point~$\mathscr{S}({\bartheta})$.
\end{assumption}

\begin{theorem}\label{thm:lownoise}
  Let ${\bartheta}\in \mathcal{U}$ such that \Cref{ass:transversality} holds.
  Then the SDP relaxation~\eqref{eq:sdp} correctly recovers the global minimizer of~\eqref{eq:stls} whenever $\btheta$ is close enough to ${\bartheta}$.
\end{theorem}

We will show in \Cref{s:nontangential} that the role of \Cref{ass:transversality}
is to ensure that the intersection $\Mdef \cap \mathcal{L}$ is a regularly defined smooth manifold.
Due to the pivotal role of regularity (constraint qualifications) in optimization,
the non-tangentiality condition becomes essential for any (local or global) optimization method for~\eqref{eq:stls}.
For instance, \Cref{ass:transversality} is needed to get linear convergence in alternating projections~\cite{Andersson2013}
and to prove quadratic convergence of Newton's method~\cite{Schost2016}.

Our proof of \Cref{thm:lownoise} relies on the theory of stability of SDP relaxations recently developed in~\cite{Cifuentes2017stability}.
More generally, the proof can be extended to the case of weighted $\ell_2$-norms defined by a positive definite weight matrix.

\subsection*{Related work}

Problem~\eqref{eq:stls} has appeared in several works, often by the name STLS;
see \cite{Markovsky2007,DeMoor1993,Rosen1996,Lemmerling1999,Lemmerling2001} and the references therein.
All previous methods rely on local optimization.
Due to nonconvexity, there are no guarantees to get the global minimum.

Recent work has shifted into a generalization of~\eqref{eq:stls}
known as \emph{structured low rank approximation} (SLRA),
given by
$
  \min_{\bu} \{ \|\bu\!-\!\btheta\|^2 :
  \rank \mathscr{S}(\bu) \!\leq\! r\},
$
for some given~$r$.
Hence,~\eqref{eq:stls} corresponds to the case $r \!=\! m{-}1$.
The SLRA problem has been studied extensively~\cite{Chu2003,Schost2016,Ottaviani2014}, notably by Markovsky~\cite{Markovsky2011,Markovsky2008}.
As before, most practical algorithms rely on local optimization,
see e.g.,~\cite[\S3]{Markovsky2011}.

Convex optimization methods have been applied to SLRA.
A common heuristic method consists in replacing the rank constraint by the convex constraint $\|U\|_* \!\leq\! t$,
where $\|\cdot\|_*$ is the nuclear norm and $t$ is a tunable parameter.
This leads to a tractable SDP, see~\cite{Fazel2002}.
The rationale behind this heuristic is that the nuclear norm is the convex envelope of the rank function.
But there are no theoretical guarantees for the SLRA problem.
The resulting matrix $U$ is not necessarily optimal, may not satisfy the rank constraint, and optimality certificates cannot be recovered.

Another convex relaxation for SLRA was proposed in~\cite{Larsson2016},
see also~\cite{Carlsson2019}.
It relies on a characterization of the convex envelope for
$\|U{-}\Theta\|^2 + \mathbf{1}(\rank U \!\leq\! r)$,
where $\mathbf{1}$ denotes the indicator function.
This gives a better approximation of the SLRA problem than the nuclear norm,
as it takes into account the $\ell_2$-objective and the rank~$r$
(but ignores the affine structure~$\mathcal{S}$).
This relaxation allows to derive optimality certificates.
However, no provable guarantees or complexity estimates are known.

To the best of our knowledge, our proposed SDP is the first convex relaxation for problem~\eqref{eq:stls} with provable guarantees.
It is also the first polynomial-time method for~\eqref{eq:stls} that provides global optimality certificates.
Special instances from computer vision have received more attention~\cite{Kahl2007,Aholt2012}, as we will elaborate in \Cref{s:vision}.

A closely related problem to SLRA is the \emph{rank minimization problem} (RMP),
given by
$
\min_{\bu} \{ \|\bu\!-\!\btheta\|^2 + t \rank \mathscr{S}(\bu) \}
$
for some parameter $t$.
RMP has been widely studied in the compressing sensing community.
Several theoretical guarantees are known for both convex~\cite{Recht2010} and nonconvex methods~\cite{Jain2010}.
Both cases critically rely on the restricted isometry property (RIP),
which gives strong limitations on the affine structure~$\mathscr{S}$.
Note that RIP does not make sense in the context of SLRA, 
since it implies that there is a unique vector $\bu$ for which $\rank \mathscr{S}(\bu) \leq r$.

Since~\eqref{eq:stls} is a polynomial optimization problem,
the sum-of-squares (SOS) method~\cite{Blekherman2013} provides a hierarchy of SDP relaxations with increasing power and complexity.
Our proposed relaxation involves a positive semidefinite (PSD) matrix of size $(k{+}1)m$, and lies in between the first and the second level of the SOS hierarchy, whose sizes are $k{+}m{+}1$ and $\binom{k{+}m{+}2}{2}$.
Though SOS relaxations often perform well in practice, there are very few theoretical results explaining this behavior.
This paper contributes in this direction, by showing that our relaxation is always exact in the low noise regime.

\subsection*{Structure of the paper}

In \Cref{s:shor} we review semidefinite programming relaxations of quadratically constrained quadratic programs.
In \Cref{s:rankone} we derive our SDP relaxation for~\eqref{eq:stls}.
In \Cref{s:nontangential} we elaborate on the connection between non-tangentiality, smoothness, and constraint qualifications.
In \Cref{s:proof} we prove \Cref{thm:lownoise}.
We conclude in \Cref{s:applications} with experimental results.

\subsection*{Notation}

We use lowercase boldface for vectors,
uppercase italics for matrices,
and uppercase calligraphic font for sets.
All vectors are column vectors.
The symbol ``$\cat$'' denotes vertical concatenation of vectors or matrices.
In particular, if $\bz \in \RR^m$ and $\bv \in \RR^k$
then $(\bz \cat \bv) \in \RR^{m+k}$.
We use the $\ell_2$-norm for vectors and the Frobenius norm for matrices
unless stated otherwise.
The set $\{1,2,\dots,n\}$ is denoted as $[n]$.

\section{Shor relaxation of a QCQP}\label{s:shor}

A \emph{quadratically constrained quadratic program} (QCQP) has a well studied SDP relaxation.
We review this relaxation in this section.

Let $\SS^N$ denote the space of real symmetric matrices of size~$N\times N$.
Given matrices $C,A_1,\dots,A_n \in \SS^N$ and a vector $\bb\in \RR^n$,
consider the \emph{homogeneous} QCQP:
\begin{equation}\label{eq:qcqp}
  \tag{P}
  \begin{aligned}
    \min_{\bx\in \RR^N }\quad \bx^T C \bx
    \quad\text{ such that }\quad
    \bx^T A_i \bx=b_i\; \text{ for } i\in [n].
  \end{aligned}
\end{equation}
By homogeneous we mean that the cost function and constraints do not involve linear terms in~$\bx$.
We can rephrase the problem in terms of the matrix $X \!=\! \bx \bx^T \!\in\! \SS^N$.
Observe that $X$ is positive semidefinite (PSD), denoted $X \!\succeq\! 0$,
and has rank one.
Also note that $\bx^T C \bx = C \!\bullet\! X$,
where $C \!\bullet\! X := \trace(C X)$ denotes the trace inner product.
Then the QCQP is equivalent to
\begin{equation*}
\begin{aligned}
  \min_{X \in \SS^N }\quad C \bullet X
  \quad\text{ such that }\quad
  A_i \bullet X =b_i\; \text{ for } i\in [n],
  \quad
  X \succeq 0,
  \quad
  X \text{ rank-one}.
\end{aligned}
\end{equation*}

The \emph{Shor SDP relaxation} of~\eqref{eq:qcqp} is obtained by dropping the rank-one constraint:
\begin{equation}\label{eq:shor}
  \tag{P\textsuperscript{*}}
  \begin{aligned}
    \min_{X\in \SS^N }\quad C \bullet X
    \quad\text{ such that }\quad
    A_i \bullet X =b_i\; \text{ for } i\in [n],
    \quad
    X \succeq 0.
  \end{aligned}
\end{equation}
The conic dual of the above problem is the following SDP:
\begin{equation}\label{eq:dual}
  \tag{D}
  \begin{aligned}
    \min_{\blambda\in \RR^n }\quad \bb^T \blambda
    \quad\text{ such that }\quad
    C - \sum\nolimits_i \lambda_i A_i  \succeq 0.
  \end{aligned}
\end{equation}
It can be shown that~\eqref{eq:dual} is also the Lagrangian dual of~\eqref{eq:qcqp}.
The optimal values of the above problems always satisfy
$ \val\text{\eqref{eq:qcqp}} \!\geq\! \val\text{\eqref{eq:shor}} \!\geq\! \val\text{\eqref{eq:dual}} $.
In particular, for any feasible dual vector $\blambda$
we have that $\bb^T \blambda$ is a lower bound for $\val\text{\eqref{eq:qcqp}}$.

The SDP relaxation is \emph{exact} if
$ \val\text{\eqref{eq:qcqp}} \!=\! \val\text{\eqref{eq:shor}} \!=\! \val\text{\eqref{eq:dual}} $
and the minimizer $X^*$ of~\eqref{eq:shor} has rank one.
In this case,
the minimizer $\bx^*$ of~\eqref{eq:qcqp} can be obtained by factorizing~$X^*$,
and the corresponding dual solution $\blambda^*$ gives an \emph{optimality certificate} for~$\bx^*$.

We proceed to the \emph{inhomogeneous} case.
Given $C,A_i \!\in\! \SS^{N\!}$, $\bc,\ba_i \!\in\! \RR^{N\!}$, $\bb \!\in\! \RR^n$, consider
\begin{equation}\label{eq:qcqp_inhom}
\begin{aligned}
  \min_{\by\in \RR^N }\quad \by^T C \by + 2 \bc^T \by
  \quad\text{ such that }\quad
  \by^T A_i \by + 2 \ba_i^T \by =b_i\; \text{ for } i\in [n].
\end{aligned}
\end{equation}
We can {homogenize} the problem by introducing an auxiliary variable~$x_0$.
Let
$\bx \!=\! ( x_0 \Vert \by)  \!\in\! \RR^{N+1}$
be obtained by prepending $x_0$ to~$\by$.
We claim that~\eqref{eq:qcqp_inhom} is equivalent to the following homogeneous QCQP:
\begin{equation}\label{eq:qcqp_hom}
\begin{gathered}
  \min_{ \bx\in \RR^{N+1} }\quad
  \bx^T \tilde{C} \bx
  \quad\text{ such that }\quad
  x_0^2 = 1, \quad
  \bx^T \tilde{A}_i \bx = b_i
  \;\text{ for } i\in [n].
\end{gathered}
\end{equation}
where
$\tilde{C} \!:=\! \left(\begin{smallmatrix} 0 & c^T \\ c & C \end{smallmatrix}\right)$
and
$\tilde{A}_i \!:=\! \left(\begin{smallmatrix} 0 & \ba_i^T \\ \ba_i & A_i \end{smallmatrix}\right)$.
Note that $x_0 \!=\! \pm 1$ by the first constraint.
We may further assume that $x_0 \!=\! 1$,
as the problem is invariant under sign changes.
Plugging in $x_0 \!=\! 1$ in~\eqref{eq:qcqp_hom} leads to~\eqref{eq:qcqp_inhom},
as we claimed.

\section{Derivation of our SDP relaxation}\label{s:rankone}
In this section we present our SDP relaxation for~\eqref{eq:stls}.
Our strategy consists in phrasing~\eqref{eq:stls} as a QCQP,
and considering the associated Shor relaxation.

\subsection{The naive attempt}

For convenience, we first center the problem at zero with the change of variables
$\bv = \bu \!-\! \btheta$.
Denoting
$\mathscr{S}_{\btheta}(\bv) := \mathscr{S}(\bv {+} \theta)$,
problem~\eqref{eq:stls} becomes
\begin{equation*}
\begin{aligned}
  \min_{\bv\in \RR^k}
  \quad & \|\bv\|^2
  \quad \text{ such that }\quad
  \mathscr{S}_{\btheta}(\bv) \text{ is rank deficient}.
\end{aligned}
\end{equation*}
We can rewrite the above as a QCQP
by introducing a variable $\bz\in \RR^m$ representing a unit vector in the left kernel of~$\mathscr{S}_{\btheta}(\bv)$.
This leads to the inhomogeneous QCQP:
\begin{equation}\label{eq:kernelrepr}
\begin{aligned}
  \min_{\bz\in\RR^m, \bv\in \RR^k }\quad \bv^T \bv
  \quad\text{ such that }\quad
  \bz^T \bz = 1,\quad
  \bz^T\mathscr{S}_\btheta(\bv)=0.
\end{aligned}
\end{equation}
The above is known as the kernel representation~\cite{Markovsky2008}.
Although we can derive an SDP relaxation from~\eqref{eq:kernelrepr},
it does not reveal any information about the original problem.

\begin{lemma} \label{thm:firstrelaxation}
  The optimal value of the Shor relaxation of~\eqref{eq:kernelrepr} is zero for any~$\btheta$.
\end{lemma}
\begin{proof}
  The homogenized problem (see~\eqref{eq:qcqp_hom}) has variables
  $x_0 \!\in\! \RR$, $\bz \!\in\! \RR^m$, $\bv \!\in\! \RR^k$.
  We may partition the PSD matrix of the Shor relaxation accordingly:
  $$ X =
  \left(\begin{smallmatrix}
      X_{00}&X_{0z}&X_{0v}\\
      *&X_{zz}&X_{zv}\\
      *& *&X_{vv}.
  \end{smallmatrix}\right) \in \SS^{1+m+k}.
  $$
  We omitted the lower triangular blocks, as $X$ is symmetric.
  The Shor relaxation is
  \begin{equation*}
    \begin{aligned}
    \min_{X\in \SS^{1+m+k}}\;\;
    \trace (X_{vv})
    \quad\text{ s.t. }\quad
    X_{00} = 1, \;\;
    \trace (X_{zz}) = 1, \;\;
    \tilde{A}_i \!\bullet\! X = 0 \;\text{ for } i \!\in\! [n], \;\;
    X \succeq 0,
    \end{aligned}
  \end{equation*}
  where $\tilde{A}_i \in \SS^{1+m+n}$ is the constraint matrix associated to the $i$-th entry of $\bz^T\mathscr{S}_\btheta(\bv)$.
  We may partition the matrix $\tilde{A}_i$ in blocks as above.
  Observe that all its diagonal blocks are zero
  since the product $\bz^T\mathscr{S}_\btheta(\bv)$ is bilinear.
  Let $X^*$ be the matrix with blocks
  $$
  X_{00}^* = 1,
  \quad
  X^*_{zz} = \tfrac{1}{m} \id_m,
  \quad
  X^*_{vv} = 0,
  \quad
  X^*_{0z} = 0,
  \quad
  X^*_{0v} = 0,
  \quad
  X^*_{zv} = 0.
  $$
  Note that $X^*$ is feasible and $\trace(X_{vv}^*)=0$.
  So the optimal value of the SDP is zero.
\end{proof}

\subsection{Lifting the problem}
Our first QCQP formulation, given in~\eqref{eq:kernelrepr}, did not succeed.
We will derive another equivalent QCQP
in a higher dimensional space,
which gives rise to a better SDP relaxation.
Let
$\vone := (1 \Vert \bv) \in\RR^{k+1}$
be obtained by prepending a one to~$\bv$,
and consider the {Kronecker products}
\begin{align*}
  \by := {\bv}\otimes \bz = \vec(\bz{\bv}^T)\in \RR^{m k}, \quad
  \bx := (\bz \cat \by ) =
  \vone\otimes \bz \in \RR^{N}, \quad N := (k+1)m.
\end{align*}
We will restate problem~\eqref{eq:kernelrepr} in terms of the vector~$\bx \!\in\! \RR^N$.

Since $\mathscr{S}_{\btheta} : \RR^k \to \RR^{m\times n}$ is an affine map,
it can be written as
\begin{gather}\label{eq:Stheta}
  \mathscr{S}_\btheta(\bv)
  \,:=\,
  \mathscr{S}(\bv {+} \btheta)
  \,=\,
  A_{\btheta} + \sum_{j=1}^k v_j\, B_j
  \quad \text{ for some } \quad
  A_{\btheta}, B_1,\dots,B_k\in \RR^{m\times n},
\end{gather}
where only the first matrix depends on~$\btheta$.
Let $S_{\btheta} \!\in\! \RR^{N\times n}$ and
$\bs^{\btheta}_1,\dots,\bs^{\btheta}_n \!\in\! \RR^N$ such that
\begin{align*}
  \left(\begin{matrix}
      \bs^{\btheta}_1 & \bs^{\btheta}_2 & \cdots & \bs^{\btheta}_n
  \end{matrix}\right)
  \,:=\,
  S_{\btheta}
  \,:=\,
  (A_{\btheta} \cat B_1 \cat \cdots \cat B_k)
  \,\in\, \RR^{N\times n},
\end{align*}
so the horizontal concatenation of $\{\bs^{\btheta}_i\}$ equals the vertical concatenation of~$\{A_{\btheta},B_j\}$.
Observe that
$  \bv^T \bv \!=\! (\bv^T \bv) (\bz^T \bz) \!=\! \by^T \by$
and
$  \bz^T \mathscr{S}_\theta(\bv) \!=\! \bx^T S_{\btheta}$.
So we can rewrite~\eqref{eq:kernelrepr} as
\begin{equation}\label{eq:rankone0}
\begin{aligned}
  \min_{\bx = (\bz \Vert \by)}\;
  \by^T \by
  \quad\text{s.t.}\quad
  \bz^T \bz = 1, \quad
  \bx^T \bs^{\btheta}_i \!=\! 0
  \;\;\forall i \!\in\! [n], \quad
  \bx \text{ is a Kronecker product}.
\end{aligned}
\end{equation}

\begin{example}\label{ex:simple}
  Let $k\!=\!1$, $m\!=\!n\!=\!2$, and let
  $\mathscr{S}_{\theta} : \RR\!\to\! \RR^{2\times 2}$,
  $v\mapsto \left(\begin{smallmatrix}
      1 & v+\theta \\ v+\theta & v+\theta
  \end{smallmatrix}\right)$.
  Then
  \begin{align*}
    A_{\theta} =
    \left(\begin{smallmatrix}
        1 & \theta\\
        \theta & \theta\\
    \end{smallmatrix}\right),
    \quad
    B_1 =
    \left(\begin{smallmatrix}
        0 & 1\\
        1 & 1\\
    \end{smallmatrix}\right),
    \quad
    \bs^{\theta}_1 = (1, \theta, 0, 1)^T,
    \qquad
    \bs^{\theta}_2 = (\theta, \theta, 1, 1)^T.
  \end{align*}
\end{example}

Problem~\eqref{eq:rankone0} is not a QCQP
as it has linear equations and a Kronecker product constraint.
We may convert the linear equations into quadratics in a simple way.
Let $\{\be_j\}_{j=1}^N$ be the canonical basis of $\RR^N$.
Then
\begin{align*}
  \bx^T \bs^{\btheta}_i = 0
  \qquad\iff\qquad
  (\bx^T \bs^{\btheta}_i)\,x_j = \bx^T \bs^{\btheta}_i \be_j^T \bx = 0
  \;\;\forall j \!\in\! [N].
\end{align*}
Hence, \eqref{eq:rankone0} is equivalent to
\begin{equation}\label{eq:rankone1}
\begin{aligned}
  \min_{\bx = (\bz \Vert \by)}\;
  \by^T \by
  \quad\text{s.t.}\quad
  \bz^T \bz \!=\! 1, \quad
  \bx^T \bs^{\btheta}_i \be_j^T \bx \!=\! 0
  \;\;\forall i \!\in\! [n], j \!\in\! [N], \quad
  \bx \text{ is a Kron.\,prod.}
\end{aligned}
\end{equation}

The Kronecker product condition means that $\bx$ is the vectorization of an $m\times (k{+}1)$ matrix of rank one
($\bx = \vec(\bz \vone^T)$).
Recall that a matrix is rank-one if and only if all its $2\times 2$ minors vanish.
If we think of $\bx$ as a matrix, we may index its entries by pairs $(i,j)$ with $0 \!\leq\! i \!\leq\! k$, $1 \!\leq\! j \!\leq\! m$,
so that $x_{(i,j)} = \tilde v_i\, z_j$.
Then
\begin{equation} \label{eq:indices}
\begin{gathered}
  \bx \text{ is a Kronecker product}
  \quad \iff \quad
  x_{l_1}x_{l_2} = x_{l_3}x_{l_4} \;\text{ for all } l=(l_1,l_2,l_3,l_4) \in L, \\
  L := \{l=\bigl((i_1,j_1),(i_2,j_2),(i_1,j_2),(i_2,j_1)\bigr) :\,
  0\!\leq\! i_1\!<\!i_2\!\leq\! k,\, 1\leq\! j_1\!<\!j_2\!\leq\! m\}.
\end{gathered}
\end{equation}
Injecting the above equations in~\eqref{eq:rankone1} lead to a QCQP formulation.

\subsection{Block symmetry}

The quadratic equations in~\eqref{eq:indices} have a nice interpretation in terms of the matrix $\bx \bx^T \!\in\! \SS^N$.
They say that this matrix has repeated entries:
$[ \bx \bx^T ]_{l_1,l_2} \!=\! [ \bx \bx^T ]_{l_3,l_4}$ for $l \!\in\! L$.
The pattern of repeated entries has a very special structure.
We say that a matrix $X \!\in\! \SS^N$ is $m$-\emph{block symmetric}
if when divided into $m \!\times\! m$ blocks
we have that all $(N/m)^2$ blocks are symmetric ($m$ must divide~$N$).
Then
\begin{align*}
  (\text{ equations~\eqref{eq:indices} hold })
  \quad\iff\quad
  \bx \bx^T \in\,
  \bSS := \{ X \!\in\! \SS^N : X \text{ is $m$-block symmetric} \}.
\end{align*}

\begin{example}\label{ex:simple1}
  Let $k\!=\!1$, $m\!=\!2$, $N \!\!= 4$.
  We may index the entries of $\bx \!\in\! \RR^4$ as
  $(x_{01},x_{02},x_{11},x_{12})$.
  Reshaping $\bx$ as a $2 \!\times\! 2$ matrix we get a single minor constraint:
  $x_{01}x_{12}\!=\!x_{02}x_{11}$.
  The matrix $\bx \bx^T$ can be divided into four $m\times m$ blocks:
  \begin{align*}
    \bx \bx^T =
    \left(\begin{smallmatrix}
        X_{zz}&X_{zy}\\
        X_{zy}^T&X_{yy}
    \end{smallmatrix}\right),
    \;\;
    X_{zz} = \left(\begin{smallmatrix}
        x_{01} x_{01}&x_{01} x_{02}\\x_{01} x_{02}&x_{02} x_{02}
    \end{smallmatrix}\right),
    \;\;
    X_{zy} = \left(\begin{smallmatrix}
        x_{01} x_{11}&x_{01} x_{12}\\x_{02} x_{11}&x_{02} x_{12}
    \end{smallmatrix}\right),
    \;\;
    X_{yy} = \left(\begin{smallmatrix}
        x_{11} x_{11}&x_{11} x_{12}\\x_{11} x_{12} &x_{12} x_{12}
    \end{smallmatrix}\right).
  \end{align*}
  The minor constraint ensures that $X_{zy}$ is symmetric, and hence $\bx \bx^T \in \bSS$.
\end{example}

By the above discussion, we can rephrase~\eqref{eq:rankone1} as
\begin{equation}\label{eq:rankone11}
\begin{aligned}
  \min_{\bx = (\bz \Vert \by)}\;
  \by^T \by
  \quad\text{s.t.}\quad
  \bz^T \bz \!=\! 1, \quad
  \bx^T \bs^{\btheta}_i \be_j^T \bx \!=\! 0
  \;\;\forall i \!\in\! [n], j \!\in\! [N], \quad
  \bx \bx^T \in \bSS.
\end{aligned}
\end{equation}
The above is a QCQP, so we can derive an associated SDP relaxation.
But before that we will make one last transformation that will simplify the analysis in the next sections.

There is a simple way turn an arbitrary matrix $M \!\in\! \RR^{N\times N}$ into block symmetric form.
We first symmetrize $M$ by taking the average $\frac{1}{2}(M {+} M^T)$,
and then symmetrize each of its $m \!\times\! m$ blocks.
We let $\Sym(M)\in\bSS$ be the result of this \emph{block symmetrization}.

Since $\bx \bx^T \!\in\! \bSS$, it follows that $\bx^T \bs^{\btheta}_i \be_j^T \bx = \bx^T\Sym(\bs^{\btheta}_i \be_j^T)\bx$.
By substituting in~\eqref{eq:rankone11}
we obtain our final QCQP formulation:
\begin{equation}\label{eq:rankone2}
\qquad\begin{aligned}
  \min_{\bx = (\bz\Vert \by)}\quad
  & \by^T \by \\
  \text{s.t.}\quad
  & \bz^T \bz = 1\\
  & \bx^T\Sym(\bs^{\btheta}_i \be_j^T)\bx=0 &&\text{ for }i \in [n], j \in [N]\\
  & \bx \bx^T \in \bSS
  &&(\text{i.e., the equations in~\eqref{eq:indices} hold})
\end{aligned}
\end{equation}

\begin{example}
  Retake the case from \Cref{ex:simple}.
  There is a matrix $\Sym(\bs^{\btheta}_i \be_j^T)$ for each $i\in\{1,2\}$ and $j\in\{1,2,3,4\}$.
  The first and the last are:
  \begin{gather*}
    \Sym(\bs^{\btheta}_1 \be_1^{\scriptscriptstyle \!T})
    =\Sym\!\left(\!\begin{smallmatrix}
        1 & 0 & 0 & 0 \\
        \theta & 0 & 0 & 0 \\
        0 & 0 & 0 & 0 \\
        1 & 0 & 0 & 0
    \end{smallmatrix}\!\right)
    =\tfrac{1}{2}\Sym\!\left(\begin{array}{c|c}
        \!\!\!\begin{smallmatrix} 2&\theta\\\theta&0\end{smallmatrix}\!\! &
        \!\!\begin{smallmatrix} 0&1\\0&0\end{smallmatrix}\!\!\!\!
        \\[1pt] \hline
        \!\!\!\begin{smallmatrix} 0&0\\1&0\end{smallmatrix}\!\! &
        \!\!\begin{smallmatrix} 0&0\\0&0\end{smallmatrix}\!\!\!\!
    \end{array}\right)
    =\tfrac{1}{4}\!\left(\begin{array}{c|c}
        \!\!\!\!\begin{smallmatrix}4&\!2\theta\!\\\!2\theta\!&0\end{smallmatrix}\!\! &
        \!\!\begin{smallmatrix} 0&1\\1&0\end{smallmatrix}\!\!\!\!
        \\[1pt] \hline
        \!\!\!\begin{smallmatrix} 0&1\\1&0\end{smallmatrix}\!\! &
        \!\!\begin{smallmatrix} 0&0\\0&0\end{smallmatrix}\!\!\!\!
    \end{array}\right)\!,
    \\
    \Sym(\bs^{\btheta}_2 \be_4^{\scriptscriptstyle \!T})
    =\Sym\!\left(\!\begin{smallmatrix}
        0 & 0 & 0 & \theta \\
        0 & 0 & 0 & \theta \\
        0 & 0 & 0 & 1 \\
        0 & 0 & 0 & 1
    \end{smallmatrix}\!\right)
    =\tfrac{1}{2}\Sym\!\left(\begin{array}{c|c}
        \!\!\!\begin{smallmatrix} 0&0\\0&0\end{smallmatrix}\!\! &
        \!\!\begin{smallmatrix} 0&\theta\\0&\theta\end{smallmatrix}\!\!\!
        \\[1pt] \hline
        \!\!\!\begin{smallmatrix} 0&0\\\theta&\theta\end{smallmatrix}\!\! &
        \!\!\begin{smallmatrix} 0&1\\1&2\end{smallmatrix}\!\!\!
    \end{array}\right)
    =\tfrac{1}{4}\!\left(\begin{array}{c|c}
        \!\!\!\begin{smallmatrix} 0&0\\0&0\end{smallmatrix}\!\! &
        \!\!\begin{smallmatrix} 0&\theta\\\theta&2\theta\end{smallmatrix}\!\!\!\!
        \\[1pt] \hline
        \!\!\!\begin{smallmatrix} 0&\theta\\\theta&\!2\theta\!\end{smallmatrix}\!\! &
        \!\!\begin{smallmatrix} 0&2\\2&4\end{smallmatrix}\!\!\!\!
    \end{array}\right)\!.
    \end{gather*}
\end{example}

\subsection{SDP relaxation}

Problem~\eqref{eq:rankone2} is a homogeneous QCQP.
Its Shor relaxation is:
\begin{subequations}\label{eq:sdp}
\begin{equation}\label{eq:primalsdp}
  \boxed{
    \quad\begin{aligned}
      \min_{X}\quad
      & \trace(X_{yy}) \\
      \text{s.t.}\quad
      & \trace(X_{zz}) = 1 \\
      & \Sym(\bs^{\btheta}_i \be_j^T) \!\bullet X=0 && \text{ for } i\!\in\![n], j \in [N]\\
      & X =
      \left(\begin{smallmatrix}
          X_{zz} & X_{zy}\\
          X_{zy}^T & X_{yy}
      \end{smallmatrix}\right)
      \in \bSS\\
      & X\succeq 0
    \end{aligned}\;\;
  }
\end{equation}
Let $(\bSS)^\perp \subset \SS^{N}$ be the orthogonal complement of~$\bSS$,
which consists of the block skew-symmetric matrices.
The corresponding dual SDP is:
\begin{equation}\label{eq:dualsdp}
  \boxed{
    \quad\begin{aligned}
      \max_{\gamma,\mu,\Sigma}\quad
      & \gamma \\
      \text{s.t.}\quad
      &
      \left(\begin{smallmatrix}
          -\gamma\, \id_m & 0 \\
          0 & \id_{m k}
      \end{smallmatrix}\right)
      \,-\, \sum\nolimits_{ij} \mu_{ij} \Sym(\bs^{\btheta}_i \be_j^T) \,-\, \Sigma
      \;\succeq\; 0\\
      &
      \gamma, \mu_{ij} \in \RR\;
      (\text{for } ij\in[n]{\times}[N]),\quad
      \Sigma\in (\bSS)^\perp
    \end{aligned}\quad
  }
\end{equation}
\end{subequations}

The above pair of SDPs are our proposed convex relaxation for~\eqref{eq:stls}.
Both SDPs have the same optimal value since the dual problem is strictly feasible (Slater's condition).
Hence, the following relations hold:
$$ \val\text{\eqref{eq:stls}} = \val\text{\eqref{eq:rankone2}} \geq \val\text{\eqref{eq:primalsdp}} = \val\text{\eqref{eq:dualsdp}}. $$
Let $X^*,\gamma^*,\mu_{ij}^*,\Sigma^*$ be the optimal solutions of the primal/dual SDPs.
The relaxation is \emph{exact} if $X^*$ of~\eqref{eq:primalsdp} has rank one.
In this case the optimal values from above are all equal,
the minimizer $\bu^*$ of~\eqref{eq:stls} can be recovered from~$X^*$,
and the dual variables $\gamma^*,\mu_{ij}^*,\Sigma^*$ give an optimality certificate for~$\bu^*$.
\Cref{thm:lownoise} states that the relaxation is always exact in the low noise regime.

\begin{remark}[Weighted seminorms]\label{thm:weighted}
  Consider the variant of~\eqref{eq:stls} that uses the weighted $\ell_2$-(semi)norm
  $\|\bv\|_W \!=\! \sqrt{\bv^T W \bv}$, with $W \!\succeq\! 0$.
  Our SDP relaxation can be easily adapted to this case.
  In~\eqref{eq:primalsdp} we simply replace the objective function by
  $ (W \!\otimes\! \id_m) \bullet X_{yy} $.
  \Cref{thm:lownoise} remains valid for weighted $\ell_2$-norms (i.e., $W\succ 0$).
\end{remark}

\begin{remark}[Missing data]\label{thm:missing}
  Sometimes in practice the data vector $\btheta$ has missing entries~\cite[\S5.1]{Markovsky2011}, e.g., due to data corruption.
  We can model this situation by considering a diagonal weight matrix,
  where $W_{ii}$ is zero if $\theta_i$ is missing and one otherwise.
  The proof of \Cref{thm:lownoise} is no longer valid for seminorms.
  Nonetheless, we will see in \Cref{ex:hankelmissing} that the SDP relaxation might still behave well in practice.
\end{remark}

\section{The role of non-tangentiality}\label{s:nontangential}

In the previous section we rephrased problem~\eqref{eq:stls} as the QCQP~\eqref{eq:rankone2},
which involves variables $\bx \!=\! (\bz \Vert \by) \!\in\! \RR^N$,
with $\bz \!\in\! \RR^m$, $\by \!\in\! \RR^{m k}$.
Let $\mathbf{h}_{\btheta}(\bx)$ be the list of all quadratic constraints in~\eqref{eq:rankone2}
and let $\mathcal{X}_{\btheta}$ denote the feasible set:
\begin{equation}\label{eq:qcqpconstraints}
  \begin{gathered}
    \mathbf{h}_{\btheta} \,:=\, \left(
      h_0,\;
      h^{\btheta}_{ij}\text{ for } ij\!\in\! [n]{\times}[N],\;
      h_l\text{ for } l\!\in\! L
    \right),
    \quad\;
    \mathcal{X}_{\btheta} \,:=\,
    \left\{ \bx : \mathbf{h}_{\btheta}(\bx) \!=\! 0 \right\}
    \,\subset\, \RR^{N},
    \\
    h_0(\bx) := \bz^T \bz \!-\! 1,\quad
    h^{\btheta}_{ij}(\bx) := \bx^T\Sym(\bs^{\btheta}_i \be_j^T)\bx,\quad
    h_l(\bx) := x_{l_1}x_{l_2} \!-\! x_{l_3}x_{l_4}.
  \end{gathered}
\end{equation}
Let~$\bartheta \in \mathcal{U}$ be a fixed parameter satisfying \Cref{ass:transversality}.
This is the only parameter we consider in this section,
so we will omit the $\btheta$-subindices to simplify the notation.
Let $\barx$ be the optimal solution of the QCQP.
Since $\bartheta \in \mathcal{U}$ then the optimal value must be zero,
and hence $\barx$ has the form $(\barz \Vert 0)$.
We will show in this section that $\mathcal{X}$ is a smooth manifold nearby~$\barx$.
The non-tangentiality condition is the main ingredient behind this.

\subsection{Abadie constraint qualification}

We introduce here a more refined notion of smoothness,
which takes into account not only the geometry,
but also the algebra.

Let $\mathbf{f} : \RR^d \!\to\! \RR^l$ continuously differentiable
and $\mathcal{Y} := \{ \by \!\in\! \RR^d : \mathbf{f}(\by) {=} 0 \}$.
The \emph{Abadie constraint  qualification} (ACQ) holds at a point $\bary \!\in\! \mathcal{Y}$
if the set $\mathcal{Y}$ is a smooth manifold nearby~$\bary$
and the tangent space at $\bary$ is given by
$T_{\bary}(\mathcal{Y}) = \ker \nabla \mathbf{f}(\bary)$.

\begin{example}\label{ex:ACQfails}
  Let
  $
    \mathcal{Y} :=
    \{
      \by \!\in\! \RR^4 :
      y_1 y_4 {=} y_2 y_3,\,
      y_1 {+} y_4 {=} y_2 \!+\! y_3,\,
      y_1 {=} y_4
    \}.
  $
  The set $\mathcal{Y}$ consists of the multiples of $\bary \!:=\! (1,1,1,1)$.
  Since $T_{\bary}(\mathcal{Y}) \!=\! \spann \{\bary\}$
  and
  $ \ker \nabla \mathbf{f}(\bary) =
    \{\by : y_1 {+} y_4 {=} y_2 \!+\! y_3,\, y_1 {=} y_4 \}, $
  then ACQ does not hold at~$\bary$.
  Nonetheless, there is an alternative description of~$\mathcal{Y}$ for which ACQ holds,
  namely $\mathcal{Y} = \{\by : y_1 \!=\! y_2, y_2 \!=\! y_3, y_3\!=\! y_4\}$.
\end{example}

The above example illustrates that the ACQ property depends on the algebraic representation of the set~$\mathcal{Y}$.
Informally, ACQ states that the local behavior of $\mathcal{Y}$ is well described by the linearization of~$\mathbf{f}$.
This condition (and similar constraint qualifications) play a pivotal role in optimization \cite[\S5.1]{Bazaraa2013}.
For instance, constraint qualifications are needed to ensure the existence of Lagrange multipliers at a local minimum,
or to guarantee convergence for local optimization methods.

We may also define the ACQ property for parametric sets.
Let $\mathbf{g} : \RR^t \!\to\! \RR^d$
and $\mathcal{Y} := \image(\mathbf{g}) = \{ \by \!\in\! \RR^d : \by \!=\! \mathbf{g}(\bv) \text{ for } \bv \!\in\! \RR^t\}$.
We say that ACQ holds at $\bary \!=\! \mathbf{g}(\barv)$
if $\mathcal{Y}$ is a smooth manifold nearby~$\bary$
and $T_{\bary}(\mathcal{Y}) = \image \nabla \mathbf{g}(\barv)$.
Note that this is equivalent to saying that the implicit set $\{(\by,\bv) : \by = \bg(\bv)\}$ satisfies ACQ at $(\bary,\barv)$.

Assume now that $\mathbf{f}$ is a polynomial map.
In such a case $\mathcal{Y}$ is a real \emph{algebraic variety}
(the zero set of polynomial equations).
The ACQ property is called \emph{regularity} or \emph{nonsingularity} in algebraic geometry~\cite[\S16.6]{Eisenbud2013}.
Recall that the \emph{ideal} generated by $\mathbf{f} = (f_1,\dots,f_m)$ consists of all sums $\sum_i p_i(\by) f_i(\by)$ where each $p_i$ is a polynomial.
Changing the defining equations preserves the ACQ property whenever the generated ideal~$\langle \mathbf{f} \rangle$ is unchanged.
More generally, the following lemma allows us to understand how $\ker \nabla \mathbf{f}(\bary)$ behaves under a polynomial map.

\begin{lemma}\label{thm:ACQideal}
  Let $\mathbf{f}_1: \RR^{d_1} \!\to\! \RR^{l_1}$,
  $\mathbf{f}_2: \RR^{d_2} \!\to\! \RR^{l_2}$,
  $\phi: \RR^{d_1} \!\to\! \RR^{d_2}$ be polynomial maps.
  Assume that the ideals generated by
  $\mathbf{f}_1$ and $\mathbf{f}_2 \circ \phi$ satisfy
  $\langle \mathbf{f}_2 \circ \phi\rangle \subset \langle \mathbf{f}_1 \rangle$.
  Let $\bary_1 \!\in \RR^{d_1}$ with $\mathbf{f}_1(\bary_1) {=} 0$,
  and let $\bary_2 \!:=\! \phi(\bary_1)$.
  Then $\nabla \phi(\bary_1) (\ker \nabla \mathbf{f}_1(\bary_1)) \subset \ker \nabla \mathbf{f}_2(\bary_2)$.
\end{lemma}

\begin{proof}
  Let $\by_1$ be the variables in $\RR^{d_1}$.
  Since the entries of $(\mathbf{f}_2\circ \phi)(\by_1)$ lie in the ideal $\langle \mathbf{f}_1(\by_1) \rangle$,
  there is an $l_2 \times l_1$ matrix $H(\by_1)$ with polynomial entries such that
  \begin{align*}
    (\mathbf{f}_2 \circ \phi)(\by_1) = H(\by_1) \mathbf{f}_1(\by_1).
  \end{align*}
  By taking the gradient of the above equation and then evaluating at~$\bary_1$,
  we get
  \begin{align*}
    \nabla\mathbf{f}_2 (\bary_2) \nabla \phi(\bary_1) = H(\bary_1) \nabla \mathbf{f}_1(\bary_1).
  \end{align*}
  So if $\bv \in \ker \nabla \mathbf{f}_1(\bary_1)$ then
  $\nabla \phi(\bary_1) \bv \in \ker\nabla \mathbf{f}_2(\bary_2)$,
  as wanted.
\end{proof}

\subsection{Non-tangentiality and ACQ}
Recall the non-tangentiality condition from~\eqref{eq:nontangential}.
This is a geometric condition, independent of the algebraic description.
As shown next, non-tangentiality arises naturally when
establishing ACQ for an intersection.

\begin{lemma}\label{thm:nontangential}
  Let
  $\mathcal{Y}_1 := \{ \by : \mathbf{f}_1(\by) {=} 0 \}$,
  $\mathcal{Y}_2 := \{ \by : \mathbf{f}_2(\by) {=} 0 \}$,
  $\mathcal{Y} := \{ \by : \mathbf{f}_1(\by) {=} \mathbf{f}_2(\by) {=} 0 \}$.
  Let $\by \!\in\! \mathcal{Y}$ such that ACQ holds for both $\mathcal{Y}_1, \mathcal{Y}_2$ at~$\by$.
  Then ACQ holds for $\mathcal{Y}$ at~$\by$
  if and only if $\mathcal{Y}_1, \mathcal{Y}_2$ meet non-tangentially at~$\by$.
\end{lemma}
\begin{proof}
  As $\mathcal{Y}_i$ satisfies ACQ for $i\!=\!1,2$,
  then $ T_{\by}(\mathcal{Y}_i) \!=\! \ker \nabla f_i$.
  The lemma follows from:
  \begin{align*}
    \mathcal{Y} \text{ satisfies ACQ }
    &\quad\iff\quad
    \mathcal{Y} \text{ is smooth and }\,
    T_{\by}(\mathcal{Y}) = \ker \nabla \mathbf{f}_1(\by) \cap \ker \nabla \mathbf{f}_2(\by),
    \\
    \text{ non-tangential }
    &\quad\iff\quad
    \mathcal{Y} \text{ is smooth and }\,
    T_{\by}(\mathcal{Y}) = T_{\by}(\mathcal{Y}_1) \cap T_{\by}(\mathcal{Y}_2).
    \qedhere
  \end{align*}
\end{proof}

\begin{remark}
  The lemma also holds if we use parametric descriptions for either $\mathcal{Y}_1, \mathcal{Y}_2$.
\end{remark}

\begin{example}
  The set $\mathcal{Y}$ from \Cref{ex:ACQfails} is the intersection of
  $ \mathcal{M} := \{ \by \!\in\! \RR^4 : y_1 y_4 {=} y_2 y_3 \} $
  and
  $ \mathcal{L} := \{ \by \!\in\! \RR^4 :
    y_1 {+} y_4 {=} y_2 \!+\! y_3,\, y_1 {=} y_4 \} $.
  The determinantal variety $\mathcal{M}$ satisfies ACQ away from the origin,
  and the linear space $\mathcal{L}$ satisfies ACQ everywhere.
  As ACQ does not hold for $\mathcal{Y}$,
  then they do not intersect non-tangentially.
\end{example}

\subsection{ACQ for our problem}

Our main focus is in the variety $\mathcal{X}$ from~\eqref{eq:qcqpconstraints}.
We goal is to prove the following lemma.

\begin{lemma}\label{thm:ACQ}
  Under \Cref{ass:transversality}, we have that
  $\mathcal{X}$ satisfies ACQ at~$\barx$.
\end{lemma}

We first analyze a simpler case.
Consider the feasible set of the QCQP~\eqref{eq:kernelrepr}:
\begin{align}\label{eq:W}
  \mathcal{W} \,:=\, \{\,\bw \!=\! (\bz \Vert \bv)\, :\,
  \mathbf{g}(\bw) \!=\! 0 \}
  \,\subset\,
  \RR^{m + k},
  \qquad
  \mathbf{g}(\bw) \,:=\, (
  \bz^T \bz-1,\;
  \bz^T \mathscr{S}_{\bartheta}(\bv)
  ).
\end{align}
\Cref{thm:nontangential} allows us to show that
ACQ holds for $\mathcal{W}$ at $\barw := (\barz \Vert 0)$.

\begin{lemma}\label{thm:ACQ0}
  Under \Cref{ass:transversality}, we have that
  $\mathcal{W}$ satisfies ACQ at $\barw := (\barz \Vert 0)$.
\end{lemma}
\begin{proof}
  Consider the sets
  \begin{align*}
    \mathcal{V} := \{(\bz,A) \!:\!  \bz^{T\!} \bz {=} 1, \bz^{T\!} A {=} 0\} \subset \RR^m{\times}\RR^{m\times n\!},
    \quad
    \mathcal{L} := \image \mathscr{S}_{\bartheta} \subset \RR^{m\times n},
    \quad
    {\mathcal{L}}' := \RR^m \!\times\! \mathcal{L}.
  \end{align*}
  The set $\mathcal{W}$ is the intersection of $\mathcal{V}$ and~$\mathcal{L}'$,
  using a kernel description for $\mathcal{V}$ and a parametrization for~$\mathcal{L}$.
  We claim that both $\mathcal{V}, \mathcal{L}'$ satisfy ACQ everywhere.
  This is straightforward for $\mathcal{L'}$,
  as it is the image of an affine map.
  We proceed to~$\mathcal{V}$.
  The Jacobian matrix of the defining equations is
  $
    \left(\begin{smallmatrix}
        2 \bz^T  & 0 \\
        A^T  &\id_n\otimes \bz^T
    \end{smallmatrix}\right),
  $
  of size $(1{+}n) \!\times\! (m{+}m n)$.
  The Jacobian has full row rank on $\mathcal{V}$ since $\bz\!\neq\! 0$.
  By the implicit function theorem, $\mathcal{V}$ is a manifold of codimension~$n{+}1$.
  So the codimension of the tangent space is equal to the rank of the Jacobian.
  It follows that ACQ holds everywhere on~$\mathcal{V}$.

  We proceed to show that $\mathcal{V}$, $\mathcal{L}'$ meet non-tangentially at $(\barz,\bar A)$,
  where $\bar A := \mathscr{S}_{\bartheta}(0)$.
  This would conclude the proof because of \Cref{thm:nontangential}.
  By \Cref{ass:transversality},
  the manifold $\Mdef \!\subset\! \RR^{m\times n}$ of rank $m{-}1$ matrices
  meets $\mathcal{L}$ non-tangentially at $\bar A$,
  so that
  $
    T_{\bar A}(\Mdef \cap\nobreak \mathcal{L}) = T_{\bar A}(\Mdef) \cap \mathcal{L}.
  $
  Observe that each $A \!\in\! \Mdef$ has exactly two unit vectors $\bz$ in its left kernel.
  In a neighborhood $\mathcal{U} \subset \Mdef$ of $\bar A$,
  we can find a smooth map $\phi: \mathcal{U} \to \RR^m$,
  where $\phi(A)$ is a unit vector in the left kernel of~$A$,
  such that $\phi(\bar A) \!=\! \barz$.
  The map $\psi : A \mapsto (\phi(A),A)$ gives a local diffeomorphism of $\Mdef$ and $\mathcal{V}$ (nearby~$\bar A$).
  Moreover, $\psi$ also gives a local diffeomorphism of $\Mdef \cap \mathcal{L}$ and $\mathcal{V} \cap \mathcal{L}'$.
  Hence $\mathcal{V} \cap \mathcal{L}'$ is a smooth manifold nearby $(\bar A,\bv)$
  and
  \begin{align*}
    T_{(\barz,\bar A)}(\mathcal{V} \cap \mathcal{L}')
    \,\simeq\,
    T_{\bar A}(\Mdef \cap \mathcal{L})
    \,=\,
    T_{\bar A}(\Mdef) \cap \mathcal{L}
    \,\simeq\,
    T_{(\barz,\bar A)}(\mathcal{V}) \cap \mathcal{L}'.
  \end{align*}
  It follows that $\mathcal{V}, \mathcal{L}'$ meet non-tangentially at $(\barz,\bar A)$, as wanted.
\end{proof}

We are ready to prove \Cref{thm:ACQ}.
The intuitive idea is to show that the varieties $\mathcal{W}$ and~$\mathcal{X}$ are isomorphic,
so any ACQ point of $\mathcal{W}$ corresponds to an ACQ point of~$\mathcal{X}$.


\begin{proof}[Proof of \Cref{thm:ACQ}]
  Note that $\mathcal{X}$ is the image of $\mathcal{W}$ under the map
  \begin{align*}
    \phi : \RR^m\!\setminus\!\{0\} \times \RR^k \to \RR^{m}\!\setminus\!\{0\} \times \RR^{k m},
    \qquad
    (\bz, \bv) \mapsto (\bz, \bz \otimes \bv).
  \end{align*}
  Observe that the map $\phi$ is a closed embedding
  (it is variant of the projective Segre embedding).
  In particular, the left inverse of $\phi$ is the map
  \begin{align*}
    \tilde{\psi} :
    \RR^{m}\!\setminus\!\{0\} \times \RR^{k m} \to
    \RR^m\!\setminus\!\{0\} \times \RR^k,
    \qquad
    (\bz, \by) \mapsto \bigl(\bz, \,\|\bz\|^{-2} \cdot \textrm{Mat}(\by) \bz\bigr),
  \end{align*}
  where $\mathrm{Mat}(\by)$ is obtained by reshaping $\by$ as a $k\times m$ matrix.
  By \Cref{thm:ACQ0}, the ACQ property is satisfied for $\barw = (\barz \Vert 0) \in \mathcal{W}$.
  In particular, $\mathcal{W}$ is a manifold nearby~$\barw$.
  The restriction of $\phi$ to this manifold is also an embedding.
  It follows that $\mathcal{X}$ a is manifold nearby $\barx = \phi(\barw)$,
  and the differential $\nabla \phi(\barw)$
  gives an isomorphism $T_{\barx}(\mathcal{X}) \simeq T_{\barw}(\mathcal{W)}$.

  Recall that $\mathcal{W}, \mathcal{X}$ are defined by the polynomial equations $\mathbf{g}(\bw), \mathbf{h}(\bx)$.
  It is easy to see that the polynomials in $\mathbf{h} \circ \phi$ lie in the ideal $\langle \mathbf{g}\rangle$.
  By \Cref{thm:ACQideal}, 
  the differential $\nabla \phi(\barw)$,
  which is injective, gives an inclusion
  $\ker \nabla \mathbf{g}(\bw) \hookrightarrow \ker \nabla \mathbf{h}(\barx)$.
  Let
  \begin{align*}
    {\psi} :
    \RR^{m} \times \RR^{m k} \to
    \RR^m \times \RR^k,
    \qquad
    (\bz, \by) \mapsto (\bz, \textrm{Mat}(\by) \bz),
  \end{align*}
  which agrees with $\tilde{\psi}$ when $\|\bz\|{=}1$.
  One can check that the polynomials in $\mathbf{g} \circ \psi$ lie in the ideal $\langle \mathbf{h}\rangle$.
  By \Cref{thm:ACQideal}, we also have an inclusion
  $\ker \nabla \mathbf{h}(\bw) \hookrightarrow \ker \nabla \mathbf{g}(\barx)$,
  and hence $\ker \nabla \mathbf{h}(\bw) \simeq \ker \nabla \mathbf{g}(\barx)$.
  Using that ACQ holds for $\barw \in \mathcal{W}$,
  we get
  \begin{align*}
    \ker \nabla \mathbf{h}(\bx)
    \simeq
    \ker \nabla \mathbf{g}(\bw)
    =
    T_{\barw}(\mathcal{W})
    \simeq
    T_{\barx}(\mathcal{X}).
  \end{align*}
  We conclude that ACQ holds for $\barx \in \mathcal{X}$.
\end{proof}

\section{Exactness under low noise}\label{s:proof}

In this section we prove our main result, \Cref{thm:lownoise}.
To do so, we rely on the general framework to analyze the stability of SDP relaxations from~\cite{Cifuentes2017stability}.

\subsection{Stability of SDP relaxations}

Let
$\mathbf{h}_{\btheta}=(h_{\btheta}^1,h_{\btheta}^2,\dots,h_{\btheta}^\ell)$
be a parametric family of quadratic equations in variables
$\bx \!=\! (\bz \Vert \by)$,
$\bz \!\in\! \RR^m$, $\by \!\in\! \RR^{N-m}$,
parametrized by a vector~$\btheta \!\in\! \RR^k$.
Consider the family of QCQPs:
\begin{align}\label{eq:nearestpoint}
  \tag{$\textrm{P}_\btheta$}
  \min_{\bx = (\bz \Vert \by)} \quad
  \|\by\|^2,
  \qquad\text{ where } \qquad
  \mathcal{X}_{\btheta} \,:=\,
  \left\{ \bx : \mathbf{h}_{\btheta}(\bx) \!=\! 0 \right\}
  \,\subset\, \RR^{N}.
\end{align}
We assume that there is a parameter $\bartheta$ such that the optimal value of \refzero{} is zero.
It can be shown that for such a parameter the Shor SDP relaxation is always exact.
The next theorem, proved in~\cite{Cifuentes2017stability},
gives conditions under which the Shor relaxation of~\eqref{eq:nearestpoint} is stable nearby~$\bartheta$,
in the sense that it is also exact for any $\btheta$ sufficiently close to~$\bartheta$.

\begin{theorem}[SDP stability] \label{thm:stability}
  Consider the parametric family~\eqref{eq:nearestpoint}.
  Let $\bartheta$ be a parameter such that the minimizer of~\refzero\ has the form $\barx = (\barz \Vert 0)$.
  Assume that the following conditions hold:
  \begin{itemize}
    \item {\rm (constraint qualification)}
      ACQ holds for $\mathcal{X}_{\bartheta}$ at~$\barx$.
    \item {\rm (smoothness)} The set
      $\{(\btheta,\bx): \mathbf{h}_\btheta(\bx){=}0\}$
      is a smooth manifold nearby $(\bartheta,\barx)$.
    \item {\rm (not a branch point)} The right kernel of $\nabla_z \mathbf{h}_{\bartheta}(\barx)$ is trivial.
    \item {\rm (restricted Slater)} There exists $\blambda\in \Lambda$ such that
      $\mathcal{A}(\blambda)|_{(\barz)^\perp}\succ\nobreak 0$,
      where
      \begin{align}\label{eq:Amap}
        \Lambda :=\{\blambda\in \RR^\ell : \blambda^T \nabla \mathbf{h}_{\bartheta}(\barx) = 0\}
        \quad\text{ and }\quad
        \mathcal{A}(\blambda) := \sum\nolimits_{i=1}^{\ell} \lambda_i \nabla^2_{zz} h_{\bartheta}^i \in \SS^{m}.
      \end{align}
  \end{itemize}
  Then the Shor relaxation of~\eqref{eq:nearestpoint} is exact whenever $\btheta$ is close enough to~$\bartheta$.
\end{theorem}
\begin{proof}
  This is a special instance of~\cite[Thm.~5.1]{Cifuentes2017stability}
  in which the branch point condition is specialized to the case of nearest point problems (see~\cite[Ex.~5.12]{Cifuentes2017stability}).
\end{proof}

\begin{remark}
  The above theorem remains valid for weighted $\ell_2$-norms.
\end{remark}

The QCQP~\eqref{eq:rankone2} is a special case of~\eqref{eq:nearestpoint},
so it suffices to verify that the conditions from \Cref{thm:stability} are satisfied.
The last condition (restricted Slater) is particularly challenging,
as discussed in~\cite{Cifuentes2017stability}.
The linear space $\Lambda\subset\RR^\ell$ from~\eqref{eq:Amap}
consists of the Lagrange multipliers at~$\barx$.
So we need to prove the existence of a Lagrange multiplier vector~$\blambda$
that satisfies a certain PSD condition,
but the theorem does not mention how to find it.
We address this issue by introducing a new condition
which is easier to check and implies restricted Slater.

We proceed to explain this new condition.
Assume that some terms in the summation of~$\mathcal{A}(\blambda)$ vanish.
Let $\mathrm{I} \,\sqcup\, \mathrm{II} = [\ell]$
be a partition of the $\ell$~equations into two groups,
where the second group corresponds to the vanishing terms, i.e.,
\begin{align}\label{eq:partition}
  i\in \textrm{II}
  \quad\implies\quad
  (
  \text{ either }\quad
  \nabla^2_{zz} h^i_{\bar\theta} = 0
  \quad \text{ or }\quad
  \lambda_i \!=\! 0 \;\;\forall\, \blambda \in \Lambda
  \;).
\end{align}
We may divide the multipliers accordingly:
$\blambda = (\lamI,\lamII)$ with $\lamI \!\in \RR^{\ell_1}$, $\lamII \!\in \RR^{\ell_2}$.
Let $\alpha$ be the restriction of $\mathcal{A}$ to $\lamI$, i.e.,
\begin{align}\label{eq:restriction}
  \alpha \,:\, \RR^{\ell_1} \to\, \SS^m,
  \qquad
  \lamI \,\mapsto\, \sum\nolimits_{i \,\in\, \mathrm{I}} \lambda_i\, \nabla^2_{zz} h_{\bartheta}^i.
\end{align}
Note that $\mathcal{A}(\blambda) \!=\! \alpha(\lamI)$ for $\blambda \!\in\! \Lambda$
because of~\eqref{eq:partition}.
We can similarly partition the Jacobian $J := \nabla \mathbf{h}(\barx)$ in the form
$J = \left(\JI \Vert \JII\right)$,
with $\JI\in \RR^{\ell_1\times N}$, $\JII\in \RR^{\ell_2\times N}$.

The next lemma provides a sufficient condition for restricted Slater.
We use the notation $\symm(S) := \frac{1}{2}(S {+} S^T)$ for the \emph{symmetric part} of a matrix~$S$.

\begin{lemma}\label{thm:sufficientcondition}
  Let $\alpha,\JI,\JII$ as above.
  Consider the linear spaces
  \begin{align*}
  \mathcal{K} :=
  \JI (\ker \JII)
  \subset \RR^{\ell_1},
  \qquad
  \mathcal{V} :=
  \{\symm(\barz \bd^{T\!}) : \bd \!\in\! \RR^m \}
  \subset \SS^m.
  \end{align*}
  and let
  $\mathcal{V}^\perp = \{S : S \barz\!=\!0\} \subset \SS^m$
  be the orthogonal complement of $\mathcal{V}$.
  Assume that $\alpha(\mathcal{K}) \subset \mathcal{V}$
  and $\image \alpha \supset \mathcal{V}^\perp $.
  Then the restricted Slater condition holds.
\end{lemma}

\begin{proof}
  We view $\mathcal{A} : \Lambda \to \SS^m$ as a linear map with domain~$\Lambda$.
Let
$\pi : \Lambda \to \RR^{\ell_1}$
be the projection
$(\lamI,\lamII) \mapsto \lamI$.
Note that $\mathcal{A} = \alpha\circ \pi$ by construction.
Let $\mathcal{A}^*,\alpha^*,\pi^*$ be the adjoint operators of $\mathcal{A},\alpha,\pi$.
We claim that
\begin{align}\label{eq:Astar}
  \ker \mathcal{A}^*
  \;\subset\;
  \ker \alpha^* \oplus \alpha(\ker \pi^*)
  \qquad\text{ and }\qquad
  \ker \pi^* \;=\; \mathcal{K}.
\end{align}
By taking the orthogonal complement of the first equation, we conclude that
\begin{align*}
  \image \mathcal{A}
  \;\supset\;
  \image \alpha \cap (\alpha(\ker \pi^*))^\perp
  \;=\;
  \image \alpha \cap (\alpha(\mathcal{K}))^\perp
  \;\supset\;
  \mathcal{V}^\perp.
\end{align*}
The above relation implies the restricted Slater condition,
Hence, it suffices to show~\eqref{eq:Astar}.

Since $\mathcal{A}^* = \pi^*\circ \alpha^*$,
the first equation in~\eqref{eq:Astar} is a consequence of the following identity:
\begin{align}\label{eq:kernelComposition}
  \ker (L_2\circ L_1)
  \;=\; \ker L_1 \oplus L_1^*(\image L_1 \cap \ker L_2)
  \quad\text{ for linear maps }L_1,L_2.
\end{align}
This identity follows by applying the rank-nullity theorem ($\operatorname{domain} L = \ker L \oplus L^*(\image L)$) to the linear maps $L_1$, $L_2|_{\image L_1}$, and $L_2\circ L_1$.

It remains to show that $\ker \pi^* = \JI (\ker \JII).$
Let $i : \Lambda \to \RR^{\ell}$ be the inclusion and $\rho : \RR^{\ell} \to \RR^{\ell_1}$ be the projection $\blambda\mapsto \lamI$,
so that $\pi = \rho\circ i$.
Notice that $\ker \rho^* = 0$ since $\rho$ is surjective.
Also note that $\image i = \ker J^*$ since $\Lambda$ is defined by the equation $\blambda^T J=0$, and hence $\ker i^* = \image J$.
As $\pi^* = i^*\circ \rho^*$, the identity~\eqref{eq:kernelComposition} gives
\begin{align*}
  \ker \pi^* = \ker \rho^* \oplus \rho(\image \rho^* \cap \ker i^*)
  = \rho(\image \rho^* \cap \image J).
\end{align*}
Observe that $\image \rho^*= \RR^{\ell_1}\times \{0\}$.
It follows that $\bw\in \ker \pi^*$ if and only if there is some $\bx$ such that $\bw=\JI \bx$ and $\JII \bx = 0$, as wanted.
\end{proof}

We proceed to prove \Cref{thm:lownoise}.
In our case, the constraints $\textbf{h}_\btheta(\bx)$ are given in~\eqref{eq:qcqpconstraints}.
The variables are $\bx \!= \!(\bz \Vert \by)$, $\bz \!\in\! \RR^m$, $\by \!\in\! \RR^{m n}$.
We need to verify that the four conditions from \Cref{thm:stability} are satisfied.
We already showed in \Cref{thm:ACQ} that ACQ holds.
We proceed to analyze the remaining three conditions.

\subsection{Smoothness}

Let $\mathcal{X}_\btheta \!\subset\! \RR^N$ be the variety from~\eqref{eq:qcqpconstraints}
and let $\mathcal{W}_\btheta \!\subset\! \RR^{m+k}$ be the variety from~\eqref{eq:W}.
Consider the sets
\begin{align*}
  \mathcal{M}_X \,:=\,
  \{(\btheta,\bx) : \bx\!\in\! \mathcal{X}_\btheta\}
  \,\subset\, \RR^k{\times}\RR^N,
  \qquad
  \mathcal{M}_W \,:=\,
  \{(\btheta,\bw) : \bw\!\in\! \mathcal{W}_\btheta\}
  \,\subset\, \RR^k{\times}\RR^{m+k}.
\end{align*}
We need to show that $\mathcal{M}_X$ is a smooth manifold nearby $(\bartheta,\barx)$.
It is easier to first analyze~$\mathcal{M}_W$.
For the nominal parameter $\bartheta$,
we showed in \Cref{thm:ACQ0} that $\mathcal{W}_{\bartheta}$ satisfies ACQ at~$\barw$.
Hence, $\mathcal{W}_{\bartheta}$ is a smooth manifold nearby~$\barw$.
Since $\mathscr{S}_\btheta(\bv) \!=\! \mathscr{S}(\bv{+}\btheta)$,
then $\mathcal{W}_\btheta$ is an affine translation of~$\mathcal{W}_{\bartheta}$.
More precisely,
$ \mathcal{W}_\btheta = \mathcal{W}_{\bartheta} - \bw_\btheta$,
where $\bw_\btheta \!:=\! (0, \btheta{-}\bartheta) \!\in\! \RR^m {\times} \RR^k$.
It follows that $\mathcal{M}_W$ is a smooth manifold nearby $(\bartheta, \barw)$.

Note that $\mathcal{X}_\btheta$ is the image of~$\mathcal{W}_\btheta$ under the map
$\RR^m{\times}\RR^k \!\to\! \RR^m {\times} \RR^{m n}$,
$(\bz,\bv) \!\mapsto\! (\bz, \bz {\otimes} \bv)$.
In fact, this map gives an isomorphism of varieties,
as showed in the proof of \Cref{thm:ACQ}.
Therefore,
$\mathcal{M}_X$ is the image of $\mathcal{M}_W$
under the map
$(\btheta,\bz,\bv) \!\mapsto\! (\btheta, \bz, \bz{\otimes}\bv)$.
It follows that $\mathcal{M}_X$ is a smooth manifold nearby $(\bartheta, \barx)$.

\subsection{Branch point}
From now on we will only focus on the nominal parameter~$\bartheta$, and hence we will no longer write the~$\btheta$-subindices.
In order to show that $\barx$ is not a branch point, we will first provide an explicit formula for the Jacobian $\nabla \mathbf{h}(\barx)$.

We will identify $\RR^N$ with the tensor product $\RR^{k+1}\!\otimes \RR^m$.
In particular, use the following bases for $\RR^{k+1}$, $\RR^m$, and $\RR^N$:
\begin{equation*}
\begin{gathered}
  \{{\bbf}_{j_1}\}_{j_1=0}^k \,\subset\, \RR^{k+1}, \;\;
  \{ \bd_{j_2}\}_{j_2=1}^m \,\subset\, \RR^m \;\;
  \text{are the canonical bases,}\\
  \{\be_{j}\}_{j} \,:=\, \{\bbf_{j_1}\otimes \bd_{j_2}\}_{j=(j_1,j_2)}
  \,\subset\, \RR^{k+1}\!\otimes \RR^m \,=\, \RR^N.
\end{gathered}
\end{equation*}
Recall from~\eqref{eq:Stheta} that the map $\mathscr{S}_{\bartheta}$
can be written in terms of matrices
$A, B_1,\dots,B_k \!\in\! \RR^{m\times n}$.
For $i \!\in\! [n]$,
let $\ba_i,\bb_{i 1},\dots,\bb_{i k} \!\in\! \RR^m$
denote the $i$-th columns of $A, B_1,\dots,B_k$.
Observe that $\bs_i = (\ba_i \Vert \bb_{i1} \Vert \dots \Vert \bb_{ik})$.
In terms of the bases from above, we have
\begin{align} \label{eq:ai}
\ba_i,\bb_{i 1},\dots,\bb_{i k}\in\RR^m
\quad\text{ are such that }\quad
\bs_i = \bbf_0 \otimes \ba_i + \sum_{t =1}^k \bbf_{t} \otimes \bb_{i t} \in \RR^N.
\end{align}

The next lemma gives a formula for (some of) the rows of $\nabla \mathbf{h}(\barx)$.

\begin{lemma}\label{thm:limitJ}
  Let $\barx=(\barz,0)$,
  and consider the constraints from~\eqref{eq:qcqpconstraints}.
  Then
  $$\nabla h_0(\barx) = 2\bbf_0\otimes \barz,$$
  and for any $i$ and $j=(j_1,j_2)$ we have
  \begin{align*}
    \nabla h_{ij}(\barx)= \begin{cases}
      \frac{1}{2}\barz_{j_2}(\bbf_{j_1}\otimes \ba_i),& \text{ if }j_1\neq 0
      \\
      \barz_{j_2}(\bbf_0\otimes \ba_i) + \sum_{t =1}^k \bbf_{t}\otimes \symm(\bb_{i t} \bd_{j_2}^T)\barz,
       & \text{ if }j_1= 0
    \end{cases}
  \end{align*}
  where $\symm$ stands for the symmetric part.
\end{lemma}
\begin{proof}
  Since $\barx=\bbf_0\otimes \barz$, then $\nabla h_0(\barx) = 2 \barx = 2\bbf_0\otimes \barz$.
  To compute $\nabla h_{ij}$ we make use of the following identity:
  {
  \begin{align*}
    \Sym\bigl((\!\bbf{\otimes} \ba)(\bg{\otimes} \bb)^{\!T} \bigr)\,(\!u{\otimes} \bz\!) =
    \symm(\bbf \bg^{T\!})u\otimes \symm(\ba \bb^{T\!})\bz
    \;\;\text{for}\;\; \ba,\bb,\bz\!\in\!\RR^{m\!},\, \bbf,\bg,\bu\!\in\! \RR^{k{+}1\!}.
  \end{align*}
  }\!\!\!
  The above identity can be checked by straightforward manipulation.
  It follows that
  \begin{align*}
    2\nabla h_{ij}(\barx)
    &= 4 \symm(\bbf_0\bbf_{j_1}^T)\bbf_0\otimes \symm(\ba_i \bd_{j_2}^T)\barz
    + 4 \sum_{t=1}^k \symm(\bbf_{t} \bbf_{j_1}^T)\bbf_0\otimes \symm(\bb_{i t} \bd_{j_2}^T)\barz
    \\[-3pt]
    &= \barz_{j_2}(\bbf_{j_1}+\bbf_0(\bbf_{j_1}^T \bbf_0))\otimes \ba_i
    + 2 \sum_{t=1}^k ( \bbf_{j_1}^T \bbf_0)\bbf_{t}\otimes \symm(\bb_{i t} \bd_{j_2}^T)\barz.
  \end{align*}
  Reducing the above expression we get the formula in the lemma.
\end{proof}

In order to make the following proofs more explicit,
we will introduce a new assumption, which is \emph{weaker} than \Cref{ass:transversality}.

\begin{assumption} \label{ass:spanai}
  The vectors $\{\ba_i\}_i$ as in~\eqref{eq:ai} span $(\barz)^\perp$.
\end{assumption}
\begin{remark}[Assum1 $\!\implies\!$ Assum2]
  Recall that
  $\mathscr{S}_{\bar\btheta}(0) = A = (\ba_1\;\,\ba_2\;\,\cdots\;\,\ba_n)$.
  Under \Cref{ass:transversality}, this matrix has rank $m{-}1$.
  Since $\barz^T \mathscr{S}_{\bar\btheta}(0)=0$, \Cref{ass:spanai} holds.
\end{remark}

We now verify the branch point condition.

\begin{lemma} \label{thm:branchpoint}
  Under \Cref{ass:spanai}, the right kernel of $\nabla_z \mathbf{h}(\barx)$ is trivial.
\end{lemma}
\begin{proof}
  Let $\bzeta \!\in\! \RR^m$ be in the right kernel.
  Then $\nabla_z h_0(\barx) \bzeta \!=\! 0$ and $\nabla_z h_{ij}(\barx) \bzeta \!=\! 0$ for all $i,j$.
  By \Cref{thm:limitJ}, we have $\barz^T \bzeta = 0$ and
  $\barz_{j_2} (\ba_i^T \bzeta) = 0$ for all $i,j_2$.
  Since $\|\barz\|=1$, some $\barz_{j_2}$ is nonzero.
  Then $\bzeta$ is orthogonal to each of $\barz,\ba_1,\dots,\ba_n$, so it must be zero.
\end{proof}

\subsection{Restricted Slater}\label{s:secondorder}

In order to use \Cref{thm:sufficientcondition},
we partition the equations into the two groups shown in \Cref{tab:multipliers}.
Group~$\mathrm{I}$ corresponds to the first column
and group~$\mathrm{II}$ to the remaining three columns.
The table also gives the multipliers associated to each constraint.
The complete multiplier vector $\blambda$ consists of all $\mu_{ij}, \sigma_l, \gamma$.

\begin{table}[htbp]
  \centering
  \caption{Partition of the equations into two groups.}
    \begin{tabular}{ccccc}
    \toprule
    group   & I     & II    & II   & II \\
    \midrule
    multiplier & $\mu_{ij}$ & $\mu_{ij}$ & $\sigma_l$ & $\gamma$ \\
    equation & $h_{ij} = \bx^T\Sym(\bs_i \be_{j}^T)\bx$ & $h_{ij} = \bx^T\Sym(\bs_i \be_j^T)\bx$ & $h_{l} = x_{l_1}x_{l_2} \!-\! x_{l_3}x_{l_4}$ & $h_0 = \bz^T \bz \!-\! 1$ \\
    indices & $i\in [n],\, j\in \{0\}\times [m]$ & $i\in[n],\, j\in [k]\times [m]$ & $l \in L$ &  \\
    \bottomrule
    \end{tabular}
  \label{tab:multipliers}
\end{table}

Let us see that group~$\mathrm{II}$ satisfies~\eqref{eq:partition}.
It can be checked that $ \nabla^2_{zz} h \!=\! 0 $
for the equations $h$ in the second and third column.
As for the fourth column,
let $\blambda \!\in\! \Lambda$ and let us show that $\gamma \!=\! 0$.
Since $\blambda \!\in\! \Lambda$, then
$\blambda^T \nabla \mathbf{h}(\barx) \barx = 0$.
Note that all the equations, except for $h_0$, are homogeneous,
so they satisfy $ \nabla h(\barx) \barx = 2 h(\barx) = 0$.
It follows that
$\blambda^T \nabla \mathbf{h}(\barx) \barx
= \gamma\, \nabla h_0(\barx)\barx = 2 \gamma$,
and hence $\gamma\!=\!0$.

We proceed to compute the linear map $\alpha$ from~\eqref{eq:restriction}.
To do so, let
\begin{gather*}
  \{\bbf_{j_1}\}_{j_1=0}^k\subset \RR^{k+1}, \;\;
  \{ \bd_{j_2}\}_{j_2=1}^m \subset \RR^m, \;\;
  \{\bg_i\}_{i=1}^n\subset \RR^n \;\;
  \text{be the canonical bases.}
\end{gather*}
Group~$\mathrm{I}$ consists of $\ell_1 \!=\! n \!\times\! m$ equations:
$h_{ij}$ for $i \!\in\! [n]$ and $j \!\in\! \{0\} {\times} [m]$.
It follows from \Cref{thm:limitJ} that
$\nabla^2_{zz} h_{ij} = \tfrac{1}{2} \symm(\ba_i \bd_{j_2}^T)$.
We may view the multiplier vector $\lamI$ as an element of
$ \RR^n \!\otimes\! \RR^m $.
Then the linear map $\alpha$ satisfies
\begin{align}\label{eq:restriction2}
  \alpha : \RR^n{\otimes} \RR^m \to \SS^m,
  \qquad
  \bg_{i} {\otimes} \bd \mapsto \tfrac{1}{2} \symm(\ba_{i} \bd^T)
  \quad\text{ for any } i\!\in\![n],\; \bd\!\in\! \RR^m.
\end{align}

We now compute the linear space
$\mathcal{K} := \JI (\ker \JII) \subset \RR^n \!\otimes\! \RR^m$.

\begin{lemma}\label{thm:J1kerJ2}
  Under \Cref{ass:spanai}, then $\JI(\ker \JII) \subset \RR^n\otimes \{\barz\}$.
\end{lemma}
\begin{proof}
  Let
  $\mathcal{L}_1 := \RR^{k+1}\otimes \{\barz\}$,
  $\mathcal{L}_2 := \{\bbf_0\}\otimes \RR^m$.
  We will first show that
  \begin{align} \label{eq:kerJII}
    \ker \JII
    \quad\subset\quad
    (\bbf_0^\perp \!\otimes\! \{\barz\}) \oplus (\{\bbf_0\} \!\otimes\! \barz^{\perp})
    \quad\subset\quad
    \mathcal{L}_1 \oplus \mathcal{L}_2.
  \end{align}
  The second containment is clear.
  We proceed to prove the first one.
  Given $\bw\in \ker \JII$, we can write it in the form
  \begin{align*}
    \bw = \bl_1 \otimes \barz + \bbf_0\otimes \bl_2 + r \bbf_0\otimes \barz + \bbf\otimes \bzeta,
    \qquad\text{ for some }
    r\in \RR,\;
    \bl_1, \bbf\in \bbf_0^\perp,\;
    \bl_2, \bzeta\in \barz^{\perp}.
  \end{align*}
  It suffices to show that $r = 0 $ and $\bbf\otimes \bzeta = 0$.
  Since $\bw\in \ker \JII$, then
  $\nabla h_0(\barx) \bw = 0$ and also
  $\nabla h_{ij}(\barx) \bw = 0$
  for $j=(j_1,j_2)$ with $j_1> 0$.
  \Cref{thm:limitJ} gives formulas for these gradients.
  The equation $\nabla h_0(\barx) \bw = 0$ says that $r= 0$.
  Assume by contradiction that $\bbf\otimes \bzeta\neq 0$.
  The equations $\nabla h_{ij}(\barx) \bw = 0$ imply
  \begin{align*}
    \barz_{j_2} (\bbf^T \bbf_{j_1})(\bzeta^T \ba_i) = 0, \qquad \text{ for all } i,j_1,j_2.
  \end{align*}
  Since $\|\barz\|=1$ then some $\barz_{j_2}$ is nonzero.
  Also note that some $\bbf^T \bbf_{j_1}\neq 0$ since $\bbf\in \bbf_0^\perp \setminus \{0\}$.
  Similarly, some $\bzeta^T \ba_i\neq 0$ since $\barz^{\perp}= \spann\{\ba_i\}_i$.
  This is a contradiction.
  Hence, \eqref{eq:kerJII} holds.

  It remains to show that $\JI (\mathcal{L}_1)\subset \mathcal{H}$, $\JI(\mathcal{L}_2)\subset \mathcal{H}$,
  where $\mathcal{H} :=\RR^n\otimes \{\barz\}$.
  The rows of $\JI$ are $\nabla h_{ij}(\barw)$, where $j=(j_1,j_2)$, $j_1=0$.
  An explicit formula is given in \Cref{thm:limitJ}.
  Let $\bw = \bbf'\otimes \barz\in \mathcal{L}_1$, and let us see that $\JI \bw \in \mathcal{H}$.
  The $ij$-th entry of $\JI \bw$ is
  \begin{align*}
    (\bbf'\otimes \barz)^{T}
    ( \barz_{j_2}&(\bbf_0\otimes \ba_i) + \sum_{t=1}^k \bbf_{t}\otimes \symm(\bb_{i t} \bd_{j_2}^T)\barz)
    \\[-3pt]
    &=  \sum_{t=1}^k (\bbf'^T \bbf_{t})( \barz^T \symm(\bb_{i t} \bd_{j_2}^T)\barz )
    =  \sum_{t=1}^k \barz_{j_2}(\bbf'^T \bbf_{t})(\bb_{i t}^T \barz)
  \end{align*}
  where we used that $\ba_i^T \barz=0$.
  Then,
  \begin{align*}
    \JI \bw
    = \sum_{i,j_2,t} \barz_{j_2}( \bbf'^T \bbf_{t})( \bb_{i t}^T \barz) (\bg_i\otimes \bd_{j_2})
    = \sum_{i,t} ( \bbf'^T \bbf_{t})( \bb_{i t}^T \barz) (\bg_i\otimes \barz) \in \mathcal{H}.
  \end{align*}
  Consider now $\bw = \bbf_0\otimes \bzeta \in \mathcal{L}_2$.
  The $ij$-th entry of $\JI \bw$ is
  \begin{align*}
    (\bbf_0\otimes \bzeta)^{T}( \barz_{j_2}(\bbf_0\otimes \ba_i)
    + \sum_{t=1}^k \bbf_{t}\otimes \symm(\bb_{i t} \bd_{j_2}^T)\barz)
    =  \barz_{j_2}( \ba_i^T \bzeta)
  \end{align*}
  and thus
  $\JI \bw = \sum_{i} (\ba_i^T \bzeta) \bg_i\otimes \barz \in \mathcal{H}.  $
\end{proof}

We are ready to show the restricted Slater condition.

\begin{lemma}\label{thm:distincteigen0}
  Under \Cref{ass:spanai}, then the conditions from \Cref{thm:sufficientcondition} are satisfied.
  Hence, the restricted Slater condition holds.
\end{lemma}
\begin{proof}
  The formula for $\alpha$ from~\eqref{eq:restriction2}
  implies that
  \begin{align*}
    \image \alpha \,=\,
    \left\{ \symm( \bzeta\, \bd^T ) :
    \bzeta \!\in\! (\barz)^\perp, \bd \!\in\! \RR^m \right\},
    \quad
    \alpha(\RR^n \!\otimes\! \{ \barz \} )
    \,=\,
    \left\{ \symm(\bzeta\, \barz^T ) :
    \bzeta \!\in\!  (\barz)^\perp) \right\}.
  \end{align*}
  The first equation implies that $\image \alpha \supset \mathcal{V}^\perp$.
  Since $\mathcal{K} \subset \RR^n {\otimes} \{ \barz \}$ by \Cref{thm:J1kerJ2},
  then the second equation implies that
  $\alpha( \mathcal{K} ) \subset \mathcal{V}$.
  So the assumptions of \Cref{thm:sufficientcondition} hold.
\end{proof}

\section{Applications}\label{s:applications}

We implemented our SDP relaxation using \cvx~\cite{cvx}.
The code is available in the personal website of the author:
\url{http://www.mit.edu/~diegcif/}.
Although not the main focus of this paper,
our implementation allows problems with complex matrices,
weighted norms, and missing entries.
We proceed to illustrate the performance of our relaxation in various applications.
The code needed to replicate the experiments is also available in the personal website of the author.
We used the SDP solver \mosek with the default parameters for all experiments.

\subsection{Hankel structure}

Consider the problem of computing the nearest rank deficient Hankel matrix.
This structure is prevalent in applications from systems theory and control,
due to the fact that a superposition of $r$~exponential signals corresponds to an infinite Hankel matrix of rank~$r$.
By suitably choosing a finite submatrix we obtain a rank deficient Hankel matrix.
Some concrete applications are approximate realization, system identification, noisy deconvolution, and stochastic realization~\cite{Lemmerling2001,Markovsky2011,Markovsky2008}.

For our first experiment we consider the Hankel approximation problem in an abstract setting (without referencing any application),
and illustrate the ``global'' behavior of our method.
By global we mean that the data vector $\btheta$ is sampled from the uniform distribution on the unit sphere.
The left of \Cref{tab:hankel} illustrates the performance of our SDP relaxation for small values of~$m,n$.
The numbers in the table indicate the percentage of experiments for which SDP solved the problem exactly (returning an optimality certificate).
For each $m,n$ we used 2000 uniformly random vectors~$\btheta$.
Observe that the SDP was exact for all instances with $n\leq 6$.

For comparison, the right of \Cref{tab:hankel} illustrates the performance of local optimization using the software \slra~\cite{slra}.
We use the Levenberg-Marquardt algorithm (LMA) as suggested in the manual,
using the initialization scheme from~\cite{VanHuffel1994}
(gives better results than the default initialization).
As expected, the accuracy of local optimization is lower than that of SDP.
Moreover, even when the local method is exact, it is hard to certify that it indeed converged to the global optimal.
In particular, we can only verify that the local method succeeded for the instances for which SDP also succeeds.

\begin{table}[htb]
  \centering
  \caption{Nearest rank deficient Hankel matrix for random instances.}
  \label{tab:hankel}
  \setlength{\tabcolsep}{3pt}
    \begin{tabular}{c||*{8}{c}||*{8}{c}}
      \multicolumn{1}{c}{}&\multicolumn{8}{c}{SDP succeeds}&\multicolumn{8}{c}{LMA and SDP succeed}
    \\
    \toprule
    $m\backslash n$
    & 3 & 4 & 5 & 6 & 7 & 8 & 9 & 10
    & 3 & 4 & 5 & 6 & 7 & 8 & 9 & 10
    \\ \midrule
    3
    & 100\pct & 100\pct & 100\pct & 100\pct & 100\pct & 100\pct &  99\pct &  99\pct
    &  95\pct &  89\pct &  83\pct &  80\pct &  76\pct &  74\pct &  71\pct &  70\pct
    \\
    4
    &         & 100\pct & 100\pct & 100\pct &  97\pct &  92\pct &  85\pct &  79\pct
    &         &  93\pct &  85\pct &  79\pct &  73\pct &  66\pct &  58\pct &  52\pct
    \\
    5
    &         &         & 100\pct & 100\pct &  99\pct &  95\pct &  87\pct &  77\pct
    &         &         &  93\pct &  84\pct &  78\pct &  71\pct &  62\pct &  53\pct
    \\
    6
    &         &         &         & 100\pct & 100\pct &  98\pct &  91\pct &  81\pct
    &         &         &         &  92\pct &  81\pct &  75\pct &  66\pct &  58\pct
    \\
    7
    &         &         &         &         & 100\pct & 100\pct &  97\pct &  89\pct
    &         &         &         &         &  90\pct &  79\pct &  73\pct &  63\pct
    \\ \bottomrule
    \end{tabular}
\end{table}

The fraction of instances for which SDP is exact should diminish as $m,n\to\nobreak \infty$.
Nevertheless, SDP always behaves well in the low noise regime, as we showed in \Cref{thm:lownoise}.
The following example illustrates this behavior.

\begin{example}[Approximate realization]\label{ex:noisyhankel}
  Consider the discrete linear time invariant (LTI) system with transfer function
  $\frac{z-1}{z^2 - 1.6 z + 0.8}$.
  Let $\by=(y_1,y_2,y_3,\dots)$ be the impulse response of the system,
  and let $H_{n}(\by)$ be the $3\times n$ Hankel matrix with the first entries of~$\by$.
  The rank of $H_n(\by)$ is two (the order of the system).
  Assume now that the signal $\by$ is corrupted by a Gaussian noise, and hence $H_n(\by)$ is full rank.
  The approximate realization problem consists in finding the nearest vector $\bu$ such that $H_n(\bu)$ has rank two.
  \Cref{tab:hankelnoisy} shows the performance of SDP for $n\!=\!40$ and for different standard deviations of the noise.
  The numbers are the percentage of experiments which are solved exactly over $200$ repetitions.
  Observe that SDP solves all instances below a certain noise level, as predicted by \Cref{thm:lownoise}.
  For comparison we also show the performance of the local optimization methods LMA and BFGS.
  We use again the initialization from~\cite{VanHuffel1994}.
  As before, we can only verify that the local methods are exact for instances for which SDP is exact.
\end{example}

\begin{table}[htb]
  \centering
  \caption{Approximate realization problem.}
  \label{tab:hankelnoisy}
  \begin{tabular}{c||*{6}{c}}
    \toprule
    noise & 0.0 & 0.1 & 0.2 & 0.3 & 0.4 & 0.5
    \\ \midrule
    SDP & 100\pct & 100\pct & 100\pct &  99\pct &  98\pct &  95\pct
    \\
    LMA & 100\pct & 100\pct &  99\pct &  86\pct &  67\pct &  54\pct
    \\
    BFGS& 100\pct & 100\pct &  99\pct &  86\pct &  67\pct &  54\pct
    \\ \bottomrule
  \end{tabular}
\end{table}

\begin{example}[Missing data]\label{ex:hankelmissing}
  Retake the approximate realization problem from above,
  and assume that some entries of the vector
  $\by = (y_1,y_2,\dots,y_{n+2})$ are missing.
  As mentioned in \Cref{thm:missing},
  we might apply our SDP by choosing a degenerate weight matrix.
  \Cref{tab:hankelmissing} shows the performance of the SDP for $n=40$,
  considering two scenarios for the missing entries.
  In the left of \Cref{tab:hankelmissing} we delete $y_i$ for all $i \equiv 3 \text{ or }5 \pmod 5$, i.e., for $i=3,5,8,\dots,38,40$.
  Our relaxation is almost as accurate as when all entries are observed (\Cref{tab:hankelnoisy}).
  In the right of \Cref{tab:hankelmissing} we delete most of the entries,
  only keeping $y_i$ when $i \equiv 1 \text{ or } 2 \pmod {10}$.
  This setting is quite complicated, due to the long gaps in between observed entries.
  Nonetheless, our relaxation still behaves remarkably well.
  For comparison, we also show the performance of LMA and BFGS.
  For the initialization we first complete the data using a cubic spline interpolation
  (linear interpolation did not work well)
  and then use the algorithm from~\cite{VanHuffel1994}
\end{example}

\begin{table}[htb]
  \centering
  \caption{Approximate realization problem with missing data.}
  \label{tab:hankelmissing}
  \setlength{\tabcolsep}{3pt}
  \begin{tabular}{c||*{6}{c}||*{6}{c}}
    \multicolumn{1}{c}{}&\multicolumn{6}{c}{38\pct{} missing entries}&\multicolumn{6}{c}{76\pct{} missing entries}
    \\
    \toprule
    noise
    & 0.0 & 0.1 & 0.2 & 0.3 & 0.4 & 0.5
    & 0.0 & 0.1 & 0.2 & 0.3 & 0.4 & 0.5
    \\ \midrule
    SDP
    & 100\pct & 100\pct & 100\pct &  98\pct &  94\pct &  92\pct
    & 100\pct &  93\pct &  80\pct &  82\pct &  82\pct &  84\pct
    \\
    LMA
    & 100\pct &  11\pct &   0\pct &   0\pct &   0\pct &   0\pct
    & 100\pct &   1\pct &   0\pct &   0\pct &   0\pct &   0\pct
    \\
    BFGS
    & 100\pct & 100\pct &  92\pct &  68\pct &  52\pct &  45\pct
    & 100\pct &  38\pct &  24\pct &  20\pct &  16\pct &  17\pct
    \\ \bottomrule
  \end{tabular}
\end{table}

\subsection{Approximate GCD}

Let $d,n_1,n_2\in \NN$ and let
$\hat{f}\in\RR[t]_{n_1},\hat{g}\in \RR[t]_{n_2}$ be univariate polynomials of degrees $n_1,n_2$,
The degree-$d$ approximate GCD problem is:
\begin{align*}
  \min_{f\in\RR[t]_{n_1},\,g\in\RR[t]_{n_2}}\quad &\|f-\hat{f}\|^2 + \|g-\hat{g}\|^2,
  \quad \text{ such that }\quad
  \deg(\gcd(f,g)) \geq d,
\end{align*}
where $\|\cdot \|$ denotes the $\ell_2$-norm of the coefficient vector.
The approximate GCD problem can be restated as in~\eqref{eq:stls}.
Indeed, it is known that $\deg(\gcd(f,g))\geq d$ if and only if $\operatorname{Syl}_d(f,g)$ is rank deficient,
where $\operatorname{Syl}_d(f,g)$ is the $(k{-}2d)\times(k{-}d{-}1)$ Sylvester matrix
which is filled with the coefficients of $f,g$;
see e.g.,~\cite{Kaltofen2006}.

\begin{example}
  Consider the polynomials
  \begin{align*}
    f := (t^2 - 2) (t^4 + 2), \quad
    g := (t^2 - 2) (t^3-1), \quad
    \operatorname{gcd}(f,g)=t^2-2.
  \end{align*}
  We normalize the coefficients of $f,g$ and then corrupt them with Gaussian noise.
  \Cref{tab:gcd} illustrates the performance of SDP in computing the degree-$2$ approximate GCD.
  The numbers indicate the percentage of successful experiments over $200$ repetitions.
  Note that SDP solves all instances below a certain noise level.
  For comparison, we also show the performance of LMA and BFGS
  using the default initialization of \slra.
\end{example}

\begin{table}[htb]
  \centering
  \caption{Approximate GCD problem.}
  \label{tab:gcd}
  \begin{tabular}{c||*{6}{c}}
    \toprule
    noise & 0.0 & 0.05 & 0.1 & 0.15 & 0.2 & 0.25
    \\ \midrule
    SDP & 100\pct & 100\pct &  96\pct &  91\pct &  90\pct &  90\pct
    \\
    LMA & 100\pct & 100\pct &  88\pct &  71\pct &  60\pct &  60\pct
    \\
    BFGS& 100\pct & 100\pct &  88\pct &  72\pct &  60\pct &  60\pct
    \\ \bottomrule
  \end{tabular}
\end{table}

\subsection{Multi-view geometry}\label{s:vision}

The pinhole camera model~\cite[\S6]{Hartley2003} from computer vision
represents 3D points by vectors $\bx = (1, x_1, x_2, x_3)^T \in \RR^4$,
2D points (images) by vectors $\bu = (u_1, u_2)^T \in \RR^2$,
and cameras by matrices $P \in \RR^{3 \times 4}$.
The image of the 3D point~$\bx$ under the camera~$P$ is the following 2D point:
\begin{align*}
  \bu \,=\,  \Pi\, P\, \bx,
  \qquad\text{ where }\qquad
  \Pi: \RR^3 \to \RR^2,
  \quad
  (y_0,y_1,y_2)^T \mapsto (y_1/y_0,\, y_2/y_0)^T.
\end{align*}
Because of the above equation,
problems from multi-view geometry often involve ratios of linear terms.
The fractional programming problem from~\eqref{eq:fractional}
captures several of these problems~\cite{Kahl2008},
including the triangulation problem, the camera resection problem, and the homography estimation problem.

\begin{example}[Triangulation]\label{exmp:triangulation2}
  Given $\ell$~cameras $P_j \in \RR^{3\times 4}$ and noisy images $\btheta_j\in \RR^2$ of an unknown 3D point $\bx\in\RR^4$, the triangulation problem is
  \begin{align*}
    \min_{\bx\in\RR^4, \bu_j\in \RR^2}\quad &\sum_{j=1}^\ell\|\bu_j-\btheta_j\|^2,
    \quad\text{ such that }\quad
    \bu_j = \Pi P_j \bx\; \text{ for } j \in [\ell],
  \end{align*}
  This is a special case of~\eqref{eq:fractional} with $m\!=\!4$ and $k\!=\!2\ell$.
  Aholt, Agarwal, and Thomas proposed an SDP relaxation for this problem in~\cite{Aholt2012},
  and they showed that it is exact under low noise when the camera centers are not coplanar.
  The SDP from~\cite{Aholt2012} is smaller than ours:
  its PSD matrix is $(k{+}1)\times (k{+}1)$, as opposed to $4(k{+}1)\times 4(k{+}1)$ for ours.
  But, as we will see, ours is more precise, succeeding in almost all instances we considered.

  \Cref{tab:triangulation} compares both SDP relaxations, indicating the percentage of experiments that are solved exactly (the solution is rank-one) over $400$ repetitions.
  The left of \Cref{tab:triangulation} considers the synthetic data set described in~\cite{Aholt2012,Olsson2009}:
  the cameras are placed uniformly at random on the sphere of radius two pointing to the origin, and the points are generated uniformly at random inside the unit cube.
  For the right of \Cref{tab:triangulation} the cameras are placed uniformly on the line segment $(2,0,0)$---$(2,0,1)$, and the points are as before.
  In both cases the image size is approximately $2\times 2$ units.
  The second configuration is very problematic for the SDP from~\cite{Aholt2012}, since the camera centers are coplanar.
  On the other hand, our SDP~\eqref{eq:sdp} behaves equally well.
\end{example}

\begin{table}[htb]
  \centering
  \caption{Camera triangulation problem.}
  \label{tab:triangulation}
  \setlength{\tabcolsep}{3pt}
  \begin{tabular}{cc||*{6}{c}||*{6}{c}}
    \multicolumn{2}{c}{}&\multicolumn{6}{c}{cameras on a sphere}&\multicolumn{6}{c}{cameras on a line}
    \\
    \toprule
    &noise
    & 0.0 & 0.1 & 0.2 & 0.3 & 0.4 & 0.5
    & 0.0 & 0.1 & 0.2 & 0.3 & 0.4 & 0.5
    \\ \midrule
    \multirow{2}{*}{$\ell=3$}
    &our SDP
    &100\pct &100\pct & 100\pct & 100\pct & 100\pct & 99\pct
    &100\pct &100\pct & 100\pct & 100\pct & 100\pct & 100\pct
    \\
    &AAT~\cite{Aholt2012}
    &100\pct &93\pct & 87\pct & 81\pct & 74\pct & 68\pct
    &100\pct &0\pct & 0\pct & 0\pct & 0\pct & 0\pct
    \\
    \multirow{2}{*}{$\ell=7$}
    &our SDP
    &100\pct &100\pct & 100\pct & 100\pct & 100\pct & 100\pct
    &100\pct &100\pct & 100\pct & 100\pct & 100\pct & 99\pct
    \\
    &AAT~\cite{Aholt2012}
    &100\pct &100\pct & 99\pct & 96\pct & 88\pct & 79\pct
    &100\pct &0\pct & 0\pct & 0\pct & 0\pct & 0\pct
    \\ \bottomrule
  \end{tabular}
\end{table}

\begin{example}[Resectioning]
  Given $\ell$~points $\bx_j\in \RR^4$ and noisy images $\btheta_j\in\RR^2$ under an unknown projective camera $P \in \RR^{3\times 4}$, the resectioning problem is
  \begin{align*}
    \min_{P\in\RR^{3\times 4}, \bu_j\in\RR^2}\quad &\sum_{j=1}^\ell\|\bu_j-\btheta_j\|^2,
    \quad \text{ such that }\quad
    \bu_j = \Pi P \bx_j\; \text{ for } j \in [\ell].
  \end{align*}
  This is a special case of~\eqref{eq:fractional} with $m\!=\!12$ and $k\!=\!2\ell$.
  \Cref{tab:resectioning} illustrates the performance of our SDP~\eqref{eq:sdp} on the synthetic data set from~\cite{Olsson2009}:
  the cameras are randomly placed on the sphere of radius two,
  and the points inside the unit cube.
  Our SDP is exact for most instances below a noise level of~$0.1$.
  For comparison, we refer to~\cite[Fig~2]{Olsson2009}, where it is observed that their method cannot certify global optimality for any instance above a noise level of~$0.03$.

  Kahl and Henrion proposed an SDP relaxation for this problem in~\cite{Kahl2007}, though without any guarantees.
  \Cref{tab:resectioning} also shows the performance of their SDP (the Schur formulation of order~$3$).
  We use the heuristic method explained in~\cite{Kahl2007} of including a small multiple of the trace in the objective function ($\epsilon=0.001$).
  This promotes a rank-one solution, at the expense of a loss in optimality.
  As shown in \Cref{tab:resectioning}, the solution still had higher rank in most of the experiments performed, even with this heuristic.
\end{example}

\begin{table}[htb]
  \centering
  \caption{Uncalibrated camera resectioning.}
  \label{tab:resectioning}
  \begin{tabular}{cc||*{6}{c}}
    \toprule
    & noise & 0.0 & 0.05 & 0.1 & 0.15 & 0.2 & 0.25
    \\ \midrule
    \multirow{2}{*}{$\ell=6$}
    & our SDP
    &100\pct &82\pct & 80\pct & 80\pct & 83\pct & 84\pct
    \\
    & KH~\cite{Kahl2007}
    &100\pct &26\pct & 7\pct & 4\pct & 2\pct & 2\pct
    \\
    \multirow{2}{*}{$\ell=15$}
    & our SDP  &100\pct &100\pct & 85\pct & 59\pct & 37\pct & 25\pct
    \\
    & KH~\cite{Kahl2007}
    &100\pct &0\pct & 0\pct & 0\pct & 0\pct & 0\pct
    \\ \bottomrule
  \end{tabular}
\end{table}

\begin{remark}
  The homography estimation problem is also a special instance of~\eqref{eq:fractional},
  so we may use our SDP relaxation.
  Another relaxation was proposed in~\cite{Kahl2007}, but with no guarantees.
  A related problem is the estimation of the essential matrix of two views,
  for which an SDP relaxation that is exact under low noise was recently proposed~\cite{Zhao2019}.
\end{remark}

\section{Future work}

\subsubsection*{Theory}

Our SDP always recovers the global minimizer of~\eqref{eq:stls} in the low noise regime, as shown in \Cref{thm:lownoise}.
More generally, the theorem applies to weighted $\ell_2$-norms of the form
$\|\bv\|_W = \sqrt{\bv^T W \bv}$ with $W \!\succ\! 0$.
We can also use our relaxation for seminorms defined by a weight matrix $W \!\succeq\! 0$,
but our proof does not generalize.
Nonetheless, the experiments in \Cref{tab:hankelmissing} suggest that \Cref{thm:lownoise} might still hold in some cases.
Investigating this behavior is an interesting open problem.

Our experiments reveal that our SDP relaxation is quite resilient to noise,
outperforming state of the art methods.
It would be desirable to establish quantitative lower bounds on the noise tolerance.
The stability theory developed in~\cite{Cifuentes2017stability} includes lower bounds,
but preliminary calculations indicate they are overly pessimistic.
Deriving practical lower bounds on the noise tolerance is left for future work.

It is natural to ask whether our SDP relaxation can be extended to the general SLRA problem
i.e., when $\mathscr{S}(\bu)$ satisfies a rank~$r$ constraint.
We may derive an SDP relaxation for SLRA following similar steps as in Section~3
(start with the kernel formulation, rewrite in terms of the Kronecker product, and derive the Shor relaxation).
But this relaxation is never exact when $r \!<\! m{-}1$.
It is an interesting problem to derive a convex relaxation for SLRA with provable guarantees.

\subsubsection*{Computation}

SDPs can be solved in polynomial time with interior point methods,
but their high memory requirements limit their use in large scale problems.
Scalability is a pervasive issue phased by SDPs,
but it has been overcome in several other applications,
e.g., \cite{Yang2020,Rosen2019,Yurtsever2019,Cosse2017,Majumdar2019}.
Developing an scalable solver for our~SDP relaxation is a very important open problem.

There are several known techniques for improving efficiency,
see the survey paper \cite{Majumdar2019}.
Taking advantage of problem structure is critical 
(e.g., sparsity~\cite{Vandenberghe2014}, symmetry~\cite{Permenter2017}, geometry~\cite{Cifuentes2015}).
One often also relies on first order methods instead of interior point methods.
Two concrete examples are the Burer-Monteiro method~\cite{Burer2003}
and the sketching technique from~\cite{Yurtsever2019}.
The PSD matrix in our relaxation has size $(k{+}1)m$,
which can be problematic when both $m,k$ are large.
This is due to our Kronecker product formulation
(this corresponds to the second level of the SOS hierarchy).
Some ways to tackle this issue are proposed in~\cite{Cosse2017,Papp2019}.

We point out that our SDP relaxation might also be used in conjunction with local optimization methods
in order to provide a posteriori verification of candidate solutions.
This approach has been explored in e.g.~\cite{Yang2020,Iguchi2017},
and is often significantly more efficient than solving the SDP by itself.

\bigskip
\subsection*{Acknowledgments}

The author thanks Sameer Agarwal, Pablo Parrilo, Bernd Sturmfels, Rekha Thomas, and Andr\'e Uschmajew for helpful discussions and comments.
This work was partially done while the author was in the Max Planck Institute in Leipzig.

\bibliographystyle{abbrv}
\bibliography{refs}

\end{document}